\documentclass{amsart}
 \newtheorem{thm}{Theorem}[section]
 \newtheorem{cor}[thm]{Corollary}
 \newtheorem{lem}[thm]{Lemma}
 \newtheorem{prop}[thm]{Proposition}
 \theoremstyle{definition}
 \newtheorem{defn}[thm]{Definition}
  \newtheorem{rem}[thm]{Remark}
 \theoremstyle{remark}

 \numberwithin{equation}{section}

\usepackage{amssymb}
\usepackage{amsmath}
\usepackage{bbm,color}

\newcommand{\R}{\mathbb{R}}
\newcommand{\C}{\mathbb{C}}
\newcommand{\N}{\mathbb{N}}
\newcommand{\cA}{\mathcal{A}}
\newcommand{\cG}{\mathcal{G}}
\newcommand{\cH}{\mathcal{H}}
\newcommand{\T}{\mathcal{T}}
\newcommand{\sS}{\mathfrak{S}}
\newcommand{\sH}{\mathfrak{H}}
\newcommand{\Op}{-\Delta_{\Sigma,\alpha}}
\newcommand{\Opfree}{-\Delta_{\rm free}}
\newcommand{\p}{\varphi}
\newcommand{\sk}[1]{\left\langle #1 \right\rangle}
\newcommand\ov{\overline}
\newcommand{\li}{\left(} 
\newcommand{\re}{\right)} 
\DeclareMathOperator{\Imag}{Im}
\DeclareMathOperator{\dom}{dom}
\DeclareMathOperator{\ran}{ran}

\begin{document}

%
%
%
%
%
%
%
%
%

\title[]{Spectral theory for Schr\"odinger operators with \boldmath{$\delta$}-interactions supported on curves in \boldmath{$\R^3$}}

\author[J.~Behrndt]{Jussi Behrndt}
\address{Institut f\"ur Numerische Mathematik \\
TU Graz \\
Steyrergasse 30\\
8010 Graz\\
Austria}
\email{behrndt@tugraz.at}

\author[R.\,L.~Frank]{Rupert L. Frank}
\address{
Mathematics 253-37\\
Caltech\\
Pasadena, CA 91125\\
USA}
\email{rlfrank@caltech.edu}

\author[C.~K\"uhn]{Christian K\"{u}hn}
\address{Institut f\"ur Numerische Mathematik \\
TU Graz \\
Steyrergasse 30\\
8010 Graz\\
Austria}
\email{kuehn@tugraz.at}

\author[V.~Lotoreichik]{Vladimir Lotoreichik}
\address{Department of Theoretical Physics\\
Nuclear Physics Institute, Czech Academy of Sciences\\
250 68, \v{R}e\v{z} near Prague\\
Czechia}
\email{lotoreichik@ujf.cas.cz}

\author[J.~Rohleder]{Jonathan Rohleder}
\address{Institut f\"ur Mathematik \\
TU Hamburg \\
Am Schwarzenberg-Campus 3, Geb\"aude E\\
21073 Hamburg \\
Germany}
\email{jonathan.rohleder@tuhh.de}

\subjclass{Primary 81Q10; Secondary 35P15, 35P20, 35P25, 47F05, 81Q15}

\keywords{Schr\"odinger operator, $\delta$-interaction, spectral theory, Birman--Schwinger operator, Schatten--von Neumann ideal, 
spectral asymptotics, isoperimetric inequality, scattering matrix}


\begin{abstract}
The main objective of this paper is to systematically develop a spectral and scattering theory for selfadjoint Schr\"odin\-ger operators with $\delta$-interactions
supported on closed curves in $\R^3$. We provide bounds for the number of negative eigenvalues depending on the geometry of the curve,
prove an isoperimetric inequality for the principal eigenvalue, derive Schatten--von Neumann properties for the resolvent difference with the free Laplacian,
and establish an explicit representation for the scattering matrix.
\end{abstract}

\maketitle

\section{Introduction}

Schr\"odinger operators with singular interactions supported on sets of Lebes\-gue measure zero were suggested in the physics literature as solvable models in quantum mechanics
in~\cite{BP35, F36, KP31, LL61, T35}. They appear, e.g., in the modeling of zero-range interactions of quantum particles~\cite{CDFMT12,DFT94,M11,MF62}, in the theory of
photonic crystals~\cite{FK98}, and in quantum few-body systems in strong magnetic fields~\cite{BD06}. The mathematical investigation of their spectral and scattering
properties attracted a lot of attention during the last decades. First studies were mostly devoted to singular interactions supported on a discrete set of points,
see the monograph~\cite{AGHH} and~\cite[Chapter 5]{EH15}. Later on, singular interactions supported on more general curves, surfaces, and
manifolds gained much attention; there
is an extensive literature on Schr\"odinger operators with $\delta$-interactions supported on manifolds of codimension one, see, e.g,~\cite{AGS87,BLL13,BSS01,BEKS94,E08,EI01,EH15,EY01,EY02}
and the references therein.
Manifolds of higher codimension were first treated in \cite{BL77} in the very special case of an interaction supported on a straight line in~$\R^3$.
More general curves were considered in~\cite{BDE03, BT92, EF07, EK02, EK04, EK08, EK15, K12, K78, K83, P01, S95, T90}.

In the present paper we systematically develop a spectral and scattering theory for Schr\"odinger operators with singular interactions supported on curves in
the three-dimensional space. More specifically, for a compact, closed, regular $C^2$-curve $\Sigma \subset \R^3$ we consider
the selfadjoint Schr\"odinger operator $\Op$ in $L^2(\R^3)$, which corresponds to the formal differential expression
\begin{equation}\label{nu}
 -\Delta - \frac{1}{\alpha} \delta(\cdot - \Sigma),
\end{equation}
where $\alpha \in \R\setminus \{0\}$ is the inverse strength of interaction. The mathematically rigorous definition of $\Op$ is more involved
than in the case of, e.g., a curve in $\R^2$ or a hypersurface in $\R^3$. For our purposes an explicit characterization of the domain and action of $\Op$
is essential; here the key difficulty is to define an appropriate generalized trace map for functions which are not sufficiently regular; see Section~\ref{sec2} for the details. 
Our method is strongly inspired by~\cite{S95} and the abstract concept of boundary triples \cite{BL07,BL12,BGP08,DM91,DM95}, and can also be viewed as a special case of
the more general approach in~\cite{P01} (see Example~3.5 therein); cf. \cite{BT92, EK02, EK15, T90} for equivalent alternative definitions.

The main results of this paper deal with spectral and scattering properties of $\Op$ and extend and complement
results in \cite{BT92,E05,EF07,EHL06,EK04,K12,S95}. First we verify that the operator $\Op$ is in fact selfadjoint; along with this,
in Theorem~\ref{thm_KreinscheFormel} we establish a Krein type formula for the resolvent difference of $\Op$ and the free Laplacian $\Opfree$. Using
this formula we show that the resolvent difference
\begin{equation}\label{resdiff}
 (\Op - \lambda)^{-1} - (\Opfree - \lambda)^{-1}, \quad \lambda\in\rho(\Op) \cap \rho (\Opfree),
\end{equation}
is compact; in particular, the essential spectrum of $\Op$ equals $[0, \infty)$. Moreover, we provide a Birman--Schwinger principle
for the negative eigenvalues of $\Op$ and employ this principle for a more detailed study of these eigenvalues. In fact, in Theorem~\ref{thm:numberEV} we
show that the negative spectrum is always finite and we prove upper and lower estimates for the number of negative eigenvalues, depending on the (inverse) strength
of interaction $\alpha$ and the geometry of the curve; these results complement the estimates in~\cite{BT92, EK04, K02, K12}. In the case that $\Sigma$ is a circle
our estimates lead to an explicit formula for the number of negative eigenvalues. As a further main result, in Theorem~\ref{thm:EVmaximizer}
we prove that amongst all curves of a fixed length the principle eigenvalue of $\Op$ is maximized by the circle.
With this result we give an affirmative answer to an open problem formulated in~\cite[Section~7.8]{E08}. Our proof is inspired by related
considerations for $\delta$-interactions supported on loops in the plane in~\cite{E05, EHL06}.

Another group of results focuses on a more detailed comparison of $\Op$ with the free Laplacian. From a careful analysis of the operators involved in
the Krein type resolvent formula we obtain an asymptotic upper bound for the singular values $s_1 (\lambda) \geq s_2 (\lambda) \geq \dots$ of the resolvent
difference~\eqref{resdiff} in Theorem~\ref{thm:ResDiff},
\begin{equation}\label{skbound}
 s_k(\lambda) = O\left(\frac{1}{k^2\ln k}\right)\quad \text{as} \quad k \rightarrow + \infty.
\end{equation}
In particular, the resolvent difference in~\eqref{resdiff} belongs to the Schatten--von Neumann class $\mathfrak{S}_p$ for any $p > 1/2$;
this improves the trace class estimate in~\cite{BT92} and is in accordance with a previous observation for periodic curves in~\cite[Remark 4.1]{EF07}. Note
that, as a consequence of~\eqref{skbound}, the absolutely continuous spectrum of $\Op$ equals $[0,\infty)$ and the wave
operators for the scattering pair $\{\Opfree, \Op \}$ exist and are complete. In Theorem~\ref{nrthm} a representation of the associated scattering matrix is
given in terms of an explicit operator function which acts in $L^2 (\Sigma)$; this complements earlier investigations in \cite[Section 3]{BT92}.
Its proof relies on an abstract approach developed recently in~\cite{BMN15}.

The paper is organized as follows. In Section~\ref{sec2} we discuss in detail the mathematically rigorous definition of
the operator $\Op$. Section~\ref{secMainRes} contains all main results of this paper. Their proofs are carried out in the remainder of this paper.
In fact, Section~\ref{sec3} is preparatory and contains the analysis of the Birman--Schwinger operator.
The actual proofs of
Theorems~\ref{thm_KreinscheFormel}--\ref{nrthm} are contained in Section~\ref{sec4}. In a short appendix the
notions of quasi boundary triples and their Weyl functions from extension theory of symmetric operators are reviewed
and it is shown how the operators $\Opfree$ and $\Op$ fit into this abstract scheme.

\section*{Acknowledgements}
Jussi Behrndt, Christian K\"uhn, Vladimir Lotoreichik, and Jonathan Rohleder gratefully acknowledge financial support
by the Austrian Science Fund (FWF), project P~25162-N26. Vladimir Lotoreichik also acknowledges financial support
by the Czech Science Foundation, project 14-06818S. Rupert Frank acknowledges support through NSF grant DMS-1363432. 
The authors also wish to thank Johannes Brasche and Andrea Posilicano for helpful discussions 
and the anonymous referees for their helpful comments which led to various improvements.

\section{Definition of the operator \boldmath{$\Op$}}\label{sec2}

In this section we define the operator $\Op$ associated with the differential expression
\eqref{nu} in $L^2(\R^3)$. On a formal level we interpret the action of \eqref{nu} as 
\begin{align}\label{eq_cA_alpha}
 \cA_\alpha u := -\Delta u - \frac{1}{\alpha} u|_\Sigma \cdot \delta_\Sigma.
\end{align}
It will be shown that $\cA_\alpha$ gives rise to a selfadjoint operator in $L^2 (\R^3)$. The key difficulty in the definition of this operator is to specify a suitable domain.
Note that the Sobolev space $H^2 (\R^3)$ is not a suitable domain as $u|_\Sigma \cdot \delta_\Sigma\not\in L^2(\R^3)$ for all those $u\in H^2(\R^3)$ which
do not vanish identically on $\Sigma$. On the other hand, any proper subspace of $H^2 (\R^3)$ will turn out to be too small for $\Op$ to
become selfadjoint in $L^2(\R^3)$. Thus it is necessary to include suitable more singular elements in the domain of the operator. This requires the definition of
a generalized trace $u |_\Sigma$
for functions $u\in L^2 (\R^3)$ which are not sufficiently regular.

Let us first fix some notation. We assume that $\Sigma$ is a compact, closed, regular $C^2$-curve in $\R^3$ of length $L$ without 
self-intersections and that $\sigma : [0,L] \to \R^3$ is a $C^2$-parametrization of 
$\Sigma$ with $|\dot\sigma(s)| = 1$ for all $s\in[0,L]$. Occasionally we identify $\sigma$ with its $L$-periodic extension.
For $h\in L^2(\Sigma)$ we define the distribution $h\delta_\Sigma$ via
\begin{equation}\label{hdelta}
 \sk{h\delta_\Sigma,\p}_{-2, 2} = \int_\Sigma h(x)\overline{\p (x)}  d\sigma (x),\qquad \p\in H^2(\R^3),
\end{equation}
where $\p (x)$ is the evaluation of the continuous function $\p$ at $x 
\in \Sigma$, $\langle \cdot, \cdot \rangle_{-2, 2}$ denotes the duality between
$H^{-2} (\R^3)$ and $H^2 (\R^3)$, and $d \sigma$ denotes integration 
with respect to the arc length on~$\Sigma$. Note that it follows from 
the continuity of the restriction map $H^2 (\R^3) \ni \p \mapsto \p |_\Sigma \in L^2 (\Sigma)$ (see, e.g., \cite[Theorem 24.3]{BIN79}) that 
$h\delta_\Sigma\in H^{-2}(\R^3)$ and that $h\mapsto h\delta_\Sigma$ is a continuous mapping from $L^2(\Sigma)$ to $H^{-2}(\R^3)$. 
We will often use that $h\delta_\Sigma\in L^2(\R^3)$ if and only if $h=0$.

For $\lambda < 0$ we define the bounded operator
\begin{equation}\label{gammaop}
 \gamma_\lambda:L^2(\Sigma)\rightarrow L^2(\R^3),\qquad h\mapsto
\gamma_\lambda h=(-\Delta-\lambda)^{-1} (h\delta_\Sigma),
\end{equation}
where $-\Delta - \lambda$ is viewed as an isomorphism between $L^2(\R^3)$ and $H^{-2}(\R^3)$. In the following lemma useful 
representations of $\gamma_\lambda$ and its adjoint $\gamma_\lambda^* : L^2 (\R^3) \to L^2 (\Sigma)$ are provided. We denote 
the selfadjoint Laplacian in $L^2(\R^3)$ with domain $H^2(\R^3)$ by $\Opfree$.

\begin{lem}\label{Lemma_gammafield}
Let $\lambda < 0$. Then
\begin{equation}\label{gamma2}
 (\gamma_\lambda h) (x) = \int_\Sigma h(y)\frac{e^{-\sqrt{- \lambda}\vert x - y\vert}}{4\pi \vert x - y\vert} \, d \sigma (y)
\end{equation}
holds for almost all $x \in \R^3$ and all $h\in L^2(\Sigma)$. Moreover,
\begin{equation}\label{goodtohave}
   \gamma_\lambda^* u= \big( (\Opfree-\lambda)^{-1} u \big)|_\Sigma
\end{equation}
holds for all $u \in L^2 (\R^3)$.
\end{lem}

\begin{proof}
For $h \in L^2 (\Sigma)$ and $u\in L^2(\R^3)$ we have
\begin{align*}
  \sk{ \gamma_\lambda h , u }_{L^2(\R^3)}
  &=
  \sk{ \gamma_\lambda h ,
  (\Opfree-\lambda)(\Opfree - \lambda)^{-1} u
  }_{L^2(\R^3)}  \\
  &=
  \sk{ (-\Delta-\lambda) (\gamma_\lambda h) , (-\Delta -
  \lambda)^{-1} u }_{-2,2}  \\
  &=
  \sk{ h\delta_\Sigma , (-\Delta - \lambda)^{-1} u }_{-2,2}  \\
  & = \int_\Sigma h (y) \overline{\big( (\Opfree - \lambda)^{-1} u \big) (y)} d \sigma (y) \\
  &=
  \int_{\R^3} \int_\Sigma h (y) \frac{e^{-\sqrt{- \lambda}|x - y|}}{4\pi|x - y|} d \sigma(y) \overline{u(x)} d x,
\end{align*}
where we have used \eqref{hdelta} and the integral representation of $(\Opfree - \lambda)^{-1}$, see, e.g.,~\cite[(IX.30)]{RS75}. This proves both~\eqref{gamma2} and~\eqref{goodtohave}.
\end{proof}

The identity~\eqref{gamma2} indicates that in general the trace of $\gamma_\lambda h$ on $\Sigma$ does not exist due to the singularity of the integral kernel. This motivates the following regularization. Here and in the following we denote by $C^{0, 1} (\Sigma)$ the space of all complex-valued Lipschitz continuous functions on $\Sigma$. Moreover, for $x = \sigma(s_0) \in \Sigma$ and $\delta > 0$ let 
\begin{align}\label{eq:IDeltaSigma}
 I_\delta^\Sigma (x) = \{\sigma(s):s\in(s_0-\delta, s_0+\delta)\}
\end{align}
be the open interval in $\Sigma$ with center $x$ and length $2 \delta$. In order to define the trace of $\gamma_\lambda h$ in a generalized sense, for $\lambda \leq 0$, 
$h \in C^{0, 1} (\Sigma)$ and $x\in\Sigma$ we set
\begin{align}\label{blambdadef}
 (B_\lambda h)(x) = \lim_{\delta \searrow 0} \bigg[ \int_{\Sigma\setminus I_\delta^\Sigma(x)} h(y) \frac{e^{-\sqrt{-\lambda}|x - y|}}{4\pi|x - y|} \;d\sigma(y) + h(x) 
 \frac{\ln \delta }{2\pi} \bigg];
\end{align}
due to technical reasons the case $\lambda = 0$ is included here although $\gamma_\lambda$ is  defined for $\lambda < 0$ only. It will be shown in Proposition~\ref{Proposition_Main} that $B_\lambda$ is a well-defined, essentially selfadjoint operator in $L^2 (\Sigma)$ for each $\lambda \leq 0$ and that the domain of its 
closure $\overline{B_\lambda}$ is independent of $\lambda$. 
Note that the basic idea in the definition of $B_\lambda$ is to remove the singularity of $\gamma_\lambda h$ on $\Sigma$.
We remark that the limit in the definition of $B_\lambda$
can also be viewed as the finite part in the sense of Hadamard
of the first summand  as $\delta\searrow 0$; cf.~\cite[Chapter 5]{McL}.
A procedure of this type is frequently employed to define hypersingular integral operators.

With the help of $B_\lambda$ we can make the following definition.

\begin{defn}\label{def:genTrace}
Let $\lambda < 0$. For $h \in \dom \overline{B_\lambda}$ we define the generalized trace $(\gamma_\lambda h) \vert_\Sigma$ of $\gamma_\lambda h$ on $\Sigma$ by
\begin{equation*}
   (\gamma_\lambda h)\vert_\Sigma = \overline{B_\lambda} h \in L^2(\Sigma),\qquad h\in\dom
\overline{B_\lambda}.
\end{equation*}
Accordingly, for a function $u = u_c + \gamma_\lambda h$ with $u_c \in H^2(\R^3)$ and $h \in \dom \overline{B_\lambda}$ we define its generalized trace $u |_\Sigma$ on $\Sigma$ by
\begin{equation}\label{traceudef}
  u \vert_\Sigma = u_c |_\Sigma + (\gamma_\lambda h) |_\Sigma = u_c \vert_\Sigma + \overline{B_\lambda} h.
\end{equation}
\end{defn}

Note that $u\vert_\Sigma$ is well-defined. Indeed, the representation 
of~$u$ as a sum is unique since
$\gamma_\lambda h\in H^2(\R^3)$ implies $h=0$. Moreover, the definition 
of $u|_\Sigma$ is independent of the choice of
$\lambda < 0$; cf. Section~\ref{subsec:Well-definedness}.

Furthermore, note that the expression $\cA_\alpha$ in \eqref{eq_cA_alpha} is no longer formal,
but makes sense as we have defined the generalized trace $u|_\Sigma$.
Now we are able to define the Schr\"{o}dinger operator
$\Op$ corresponding to the differential expression
in~\eqref{nu} in a rigorous way.

\begin{defn}\label{schdef}
For $\alpha\in\R\setminus\{0\}$ the Schr\"{o}dinger operator $\Op$ in $L^2 (\R^3)$ with $\delta$-interaction 
of strength $\frac{1}{\alpha}$ supported on $\Sigma$ is defined by
\begin{equation*}
\begin{split}
 \Op u & = \cA_\alpha u= - \Delta u - \frac{1}{\alpha} u|_{\Sigma} \cdot \delta_\Sigma,\\
 \dom(\Op)& = \bigl\{u=u_c+\gamma_\lambda h: u_c\in H^2(\R^3),\, h\in\dom \overline{B_\lambda},\, \cA_\alpha u \in L^2(\R^3)\bigr\},
\end{split}
\end{equation*}
where $\lambda<0$ is arbitrary and the generalized trace $u\vert_\Sigma$ is defined in~\eqref{traceudef}.
\end{defn}

Observe that the operator $\Op$ is well-defined since $\dom \overline{B_\lambda}$ and the trace $u\vert_\Sigma$ do not depend on the choice of $\lambda$. Note also that for $\alpha = +\infty$ we formally have 
\begin{equation*} 
 -\Delta_{\Sigma,+\infty}u=-\Delta u,\qquad \dom(-\Delta_{\Sigma,+\infty}) = H^2(\R^3),
\end{equation*}
so that the Schr\"{o}dinger operator with $\delta$-interaction of strength $0$ on $\Sigma$ coincides with the free Laplacian $\Opfree$; this will be made precise in Theorem~\ref{thm_KreinscheFormel}~(ii) below.

\begin{rem}\label{KonstUmrechnen}
The definition of $\Op$ relies on the generalized trace in Definition~\ref{def:genTrace} and, thus, on the operator $B_\lambda$. 
As mentioned above, the operator $B_\lambda$ is designed in such a way that the singularity of $\gamma_\lambda h$ on $\Sigma$ is removed; this is done here 
by the term $\frac{\ln \delta }{2\pi}$. However, an alternative choice $\frac{\ln \delta }{2\pi} + c$ with an arbitrary $\delta$-independent constant $c\in\R$ can be made. 
This leads to a different operator $\Op$, which can be transformed into the operator in Definition~\ref{schdef} by adding the same constant $c$ to $\alpha$.
For instance, for $c = - \frac{\ln2}{2\pi}$ one obtains the family of operators considered in~\cite{T90}.
\end{rem}

\begin{rem}
	For a function
	$u = u_{\rm c} + \gamma_\lambda h\in \dom(-\Delta_{\Sigma,\alpha})$
	with $h\in C^{0,1}(\Sigma)$ we denote by  
	$\widehat{u}(s,\delta)$, $s\in[0,L)$, the mean value of $u$ over a circle of 
	a sufficiently small radius $\delta > 0$ centered at $\sigma(s)$
	and being orthogonal to $\Sigma$ in $\sigma(s)$. 
	According to~\cite[Remark 3]{T90} (see also~\cite{EF07, EK02}) the 
	functions
	\[
		h_0(s) 
		:= 2\pi\lim_{\delta\searrow 0} \frac{\widehat{u}(s,\delta)}{\ln(1/\delta)}
		\quad\text{and}\quad
		h_1(s) :=
		\lim_{\delta\searrow 0}\left[\widehat{u}(s,\delta) - 
		\frac{h_0(s)}{2\pi}\ln\bigg(\frac{1}{\delta}\bigg) \right]
	\]
	are well-defined and continuous on $\Sigma$
	and the function $u$
	satisfies the following boundary condition
	\[
		h_1(s) = \bigg(\alpha + \frac{\ln 2}{2\pi }\bigg) h_0(s).
	\]
	In many-body physics with zero-range interactions
	a boundary condition of this type
	is known as Skorniakov--Ter-Martirosian condition; see~\cite{ST56}
	and also~\cite{CDFMT12, MO16}.
\end{rem}

\section{Main results}\label{secMainRes}

In this section we present all main results of this paper. It will be shown that $\Op$
is selfadjoint and its spectral and scattering properties will be analyzed. This section is focused on the main statements and does not contain their proofs; these are postponed to 
Section~\ref{sec4} below. In the following we denote by $\sigma_{\rm p} (\Op)$, $\sigma_{\rm ess} (\Op)$, and $\rho (\Op)$ the point spectrum, essential spectrum, and resolvent set of $\Op$, respectively. 

In the first theorem we check that $\Op$ is a selfadjoint operator in $L^2(\R^3)$, 
prove a Birman--Schwinger principle for its negative eigenvalues and compare its resolvent to the resolvent of the free Laplacian $\Opfree$ 
in a Krein type formula, which also implies that the difference of the resolvents is compact.

\begin{thm}\label{thm_KreinscheFormel}
The Schr\"{o}dinger operator $\Op$ in $L^2(\R^3)$ in 
Definition~\ref{schdef} is selfadjoint. Moreover, the following assertions hold.
\begin{enumerate}
 \item For each $\lambda < 0$ the operator $\gamma_\lambda$ is an isomorphism between
  $\ker(\alpha - \overline{B_\lambda})$ and $\ker(\Op-\lambda)$. In particular, for each $\lambda < 0$
 \begin{align*}
  \lambda\in \sigma_{\rm p} (\Op) \quad \text{if and only if}
  \quad \alpha \in \sigma_{\rm p} ( \overline{B_\lambda}).
 \end{align*}
 \item The set $\rho(\Op) \cap (- \infty, 0)$ is nonempty and for each $\lambda \in \rho(\Op) \cap (- \infty, 0)$ the resolvent formula
 \begin{align}\label{KreinscheFormel}
  (\Op - \lambda)^{-1} = (\Opfree - \lambda)^{-1} + \gamma_{\lambda} \bigl(\alpha - \overline{B_\lambda}\,\bigr)^{-1} \gamma_{\lambda}^*
\end{align}
is valid. Furthermore, $\Op$ converges to $\Opfree$ in the norm resolvent sense as $\alpha\to + \infty$.
 \item For each $\lambda \in \rho (\Op) \cap \rho (\Opfree)$ the resolvent difference
 \begin{align}\label{eq:ResDiff}
  (\Op - \lambda)^{-1} - (\Opfree - \lambda)^{-1}
\end{align}
is compact and, in particular, $\sigma_{\rm ess} (\Op) = [0, \infty)$.
\end{enumerate}
\end{thm}

Next we investigate the resolvent difference of $\Op$ and the free Laplacian in more detail.

\begin{thm}\label{thm:ResDiff}
Let $s_1 (\lambda) \geq s_2 (\lambda) \geq \dots$ be the singular values  of the resolvent difference of $\Op$ and $\Opfree$ in~\eqref{eq:ResDiff}, counted with multiplicities. Then
\begin{align*}
  s_k (\lambda) = O \Big(\frac{1}{k^2\ln k} \Big) \quad \text{as} \quad k\to + \infty.
\end{align*}
In particular,~\eqref{eq:ResDiff} belongs to the Schatten--von Neumann ideal $\mathfrak{S}_p (L^2(\R^3))$ for each $p > 1/2$. 
\end{thm}
The logarithmic factor in the estimate for the singular values
in the above theorem is related to the fact
that the eigenvalues of $\ov{B_\lambda}$ behave asymptotically as $-\frac{\ln k}{2\pi}$,
see Proposition~\ref{Proposition_Main} (iii).

In the following theorem we show that the discrete spectrum of $\Op$ is always finite and give estimates for the number $N_\alpha$ of 
negative eigenvalues, counted with multiplicities. Let $R = \frac{L}{2 \pi}$ and define the intervals 
\begin{align*}
 I_{-1} = \bigg[ \frac{\ln(4 R)}{2\pi}, + \infty \bigg), \quad I_0 = \bigg[  \frac{\ln(4R)}{2\pi} - \frac{1}{\pi}, \frac{\ln(4R)}{2\pi} \bigg),
\end{align*}
and
\begin{align*}
 I_r = \bigg[  \frac{\ln(4R)}{2\pi} - \frac{1}{\pi} \sum_{j=1}^{r+1} \frac{1}{2j-1}, \frac{\ln(4R)}{2\pi} - \frac{1}{\pi} \sum_{j=1}^{r} \frac{1}{2j-1}  \bigg), \quad r = 1, 2, \dots,
\end{align*}
which are disjoint and satisfy $\R = \bigcup_{r = -1}^\infty I_r$. Moreover, set
\begin{align}\label{ddsigma}
 d_\Sigma = \int_0^L   \int_0^L \left| \frac{1}{4\pi|\sigma(t)-\sigma(s)|} - \frac{1}{4 \pi |\tau(t)-\tau(s)|}\right|^2   d s d t \geq 0,
\end{align}
where $\sigma$ is the parametrization of $\Sigma$ fixed in the beginning of Section~\ref{sec2} and $\tau$ denotes an arc length parametrization of a circle of radius $R$.

\begin{thm}\label{thm:numberEV}
Let $\alpha \neq 0$ and denote by $N_\alpha$ the number of negative eigenvalues of $\Op$, counted with multiplicities. If $\alpha - d_\Sigma \geq \frac{\ln (4 R)}{2 \pi}$ then $N_\alpha = 0$. Otherwise, 
 \begin{align*}
  2 r + 1 \leq N_\alpha \leq 2 l + 1,
 \end{align*}
where $r \geq -1$ and $l \geq 0$ are such that $\alpha + d_\Sigma \in I_r$ and $\alpha - d_\Sigma\in I_l$.
In particular, $N_\alpha$ is finite and the operator $\Op$
is bounded from below.
\end{thm}

In the next corollary the upper and lower bounds on the number $N_\alpha$ of negative eigenvalues in Theorem~\ref{thm:numberEV} are made more explicit. This also leads to 
an asymptotic bound $N_\alpha = e^{-2\pi\alpha+O(1)}$ as $\alpha \to -\infty$. We mention that a slightly better asymptotic bound  was obtained in \cite{EK04}.
For convenience we make a very 
small technical restriction and consider the case $\alpha + d_\Sigma < \frac{\ln(4R)}{2\pi} - \frac{1}{\pi}$ only.

\begin{cor}\label{cor:numberEVestimate}
Let $\alpha \neq 0$ be such that $\alpha + d_\Sigma < \frac{\ln(4R)}{2\pi} - \frac{1}{\pi}$ and denote by $N_\alpha$ the number of negative eigenvalues of $\Op$,
counted with multiplicities. Then the estimate
\begin{align}\label{eq:numberEVestimate}
2R c^{-1} e^{-2\pi\alpha-\gamma} -1-4(e^\frac{1}{92}-1) < N_\alpha < 2R c e^{-2\pi\alpha-\gamma} +1
\end{align}
holds, where $\gamma\approx 0.577216$ is the Euler--Mascheroni constant and $c:=e^{2\pi d_\Sigma}$.  In particular, $N_\alpha = e^{-2\pi\alpha+ O(1)}$ as $\alpha \to -\infty$.
\end{cor}

In the case where $\Sigma$ is a circle we have $d_\Sigma=0$ and hence from Theorem~\ref{thm:numberEV} and Corollary~\ref{cor:numberEVestimate} we immediately 
obtain the following explicit expressions for the number of negative eigenvalues. For a similar formula in a related context see~\cite{K12} (cf.~also~\cite{BT92}).

\begin{cor}\label{cor:circle}
Let $\Sigma$ be a circle of radius $R$ in $\R^3$, let $\alpha \neq 0$, and denote by $N_\alpha$ the number of negative 
eigenvalues of $\Op$, counted with multiplicities. If $\alpha \geq \frac{\ln (4 R)}{2 \pi}$ then $N_\alpha = 0$.
Otherwise,  
 \begin{align*}
  N_\alpha = 2r + 1, \quad \text{where}~r \geq 0~\text{is such that}~\alpha \in I_r.
 \end{align*}
If $\alpha < \frac{\ln(4R)}{2\pi}-\frac{1}{\pi}$ then the estimate
\begin{align*}
| N_\alpha - 2R e^{-2\pi\alpha-\gamma}| < 1+4(e^\frac{1}{92}-1)
\end{align*}
holds.
\end{cor}

Next, we investigate the behavior of the smallest eigenvalue of $\Op$ when varying $\Sigma$ among all curves of a given length $L$. It turns out that circles are the unique maximizers of the minimum of the spectrum $ \sigma (\Op)$ in the case that negative eigenvalues exist. 
The analog of the following theorem for curves in the two-dimensional space was shown in~\cite{E05,EHL06}. 

\begin{thm}\label{thm:EVmaximizer}
Let $\T$ be a circle in $\R^3$ of radius $R=\tfrac{L}{2\pi}$ and assume that $\Sigma$ is not a circle.
Let $\alpha < \frac{\ln (4 R)}{2 \pi}$. Then
\begin{align*}
 \min \sigma(\Op) < \min \sigma(-\Delta_{\T,\alpha}),
\end{align*}
where $-\Delta_{\T,\alpha}$ denotes the Schr\"odinger operator with $\delta$-interaction of strength
$\frac{1}{\alpha}$ supported on the circle $\T$. 
\end{thm}

Finally, we regard the pair  $\{\Opfree, \Op \}$ as a scattering system consisting of the unperturbed Laplacian $\Opfree$ and 
the singularly perturbed operator $\Op$.
The following corollary is an immediate consequence of Theorem~\ref{thm:ResDiff} and the Birman--Krein theorem \cite{BK}.

\begin{cor}
The absolutely continuous spectrum of $\Op$ is given by
\begin{equation*}
 \sigma_{\rm ac}(\Op) = [0,+\infty).
\end{equation*}
Moreover, the wave operators for the scattering pair $\{\Opfree, \Op \}$ exist and are complete.
\end{cor}

In the next theorem we express the scattering matrix of the scattering system $\{\Opfree,\Op\}$ in terms of the limits of a certain explicit 
operator function, using a result in~\cite{BMN15}; we refer to \cite{BW83,K,RS79,Y92} and Appendix~\ref{appi} 
for more details on scattering theory. For our purposes it is convenient to consider the symmetric operator $S$ in $L^2(\R^3)$ defined as
\begin{equation*}
 S u = -\Delta u,\qquad\dom S=\bigl\{u\in H^2(\R^3):u\vert_\Sigma=0\bigr\},
\end{equation*}
which turns out to be the intersection of the selfadjoint operators $\Opfree$ and $\Op$. Then $S$ is a densely defined, closed, symmetric operator with infinite defect numbers. 
Furthermore, in general $S$ contains a selfadjoint part which can be split off. More precisely, consider the closed subspace
\begin{equation*}
 \sH_1=\overline{ \text{\rm span}\, \bigcup_{\lambda\in \C\setminus[0,\infty)} (\ran(S-\lambda))^\bot }  
\end{equation*}
of $L^2(\R^3)$ and let $\sH_2=\sH_1^\bot$. Then $S$ admits the orthogonal sum decomposition
\begin{equation*}
 S=S_1\oplus S_2
\end{equation*}
with respect to the space decomposition $L^2(\R^3)=\sH_1\oplus\sH_2$, where the closed symmetric operator $S_1$ is completely non-selfadjoint or 
simple (cf.~\cite[Chapter~VII]{AG93}) in $\sH_1$ and $S_2$ is a selfadjoint operator in $\sH_2$ with purely absolutely continuous spectrum. 
In the following let $L^2(\R,d\lambda,\cH_\lambda)$ be a spectral representation of the selfadjoint operator $S_2$ in $\sH_2$; cf. \cite[Chapter 4]{BW83}.

\begin{thm}\label{nrthm}
Fix $\eta < 0$ such that $0 \in \rho (\overline{B_\eta} - \alpha)$ and define the operator function $\C \setminus [0,\infty) \ni \lambda\mapsto N(\lambda)$ by
\begin{equation}\label{nnn}
 (N(\lambda)h)(x)=\int_\Sigma h(y)\frac{e^{i\sqrt{\lambda}\vert x - y\vert}-e^{i\sqrt{\eta}\vert x - y\vert}}{4\pi \vert x - y\vert}\,d\sigma(y),
\end{equation}
where $h\in L^2(\Sigma)$ and $x\in\Sigma$.
Then the following assertions hold.
 \begin{enumerate}
  \item $\Imag N(\lambda)\in\sS_1(L^2(\Sigma))$ for all $\lambda\in\C\setminus[0,\infty)$ and the limit 
  \begin{align*}
   \Imag N(\lambda + i0) := \lim_{\varepsilon\searrow 0} \Imag N (\lambda + i\varepsilon)
  \end{align*}
  exists in $\sS_1(L^2(\Sigma))$ for a.e. $\lambda\in [0,\infty)$.
  \item The function $\lambda\mapsto N(\lambda)$, $\lambda\in\C\setminus[0,\infty)$, is a Nevanlinna function such that the limit 
  \begin{align*}
   N(\lambda + i0) := \lim_{\varepsilon\searrow 0} N(\lambda + i\varepsilon)
  \end{align*}
  exists in the Hilbert--Schmidt norm for a.e. $\lambda \in [0,\infty)$. Moreover, for a.e. $\lambda \in [0,\infty)$ 
  the operator $N(\lambda + i0)+ \overline{B_\eta}-\alpha$ is boundedly invertible.
  \item The space $L^2(\R,d\lambda,\cG_\lambda\oplus\cH_\lambda)$, where
  \begin{equation*}
   \cG_\lambda:=\overline{\ran\bigl(\text{\rm Im}\, N(\lambda+i0)\bigr)}\quad \text{for a.e.}\,\,\,\lambda\in[0,\infty),
  \end{equation*}
  forms a spectral representation of $\Opfree$.
  \item The scattering matrix $\{S(\lambda)\}_{\lambda\in\R}$ of the scattering system $\{\Opfree,\Op\}$
 acting in the space $L^2(\R,d\lambda,\cG_\lambda\oplus\cH_\lambda)$ admits the representation
  \begin{equation*}
   S(\lambda)=\begin{pmatrix} S'(\lambda) & 0 \\ 0 & I_{\cH_\lambda}\end{pmatrix} 
  \end{equation*}
  for a.e. $\lambda\in [0,\infty)$, where
  \begin{equation*}
   S'(\lambda) = I_{\cG_\lambda} - 2 i \sqrt{\text{\rm Im}\, N(\lambda+i0)}\,\bigl(N(\lambda+i0) + \overline{B_\eta}-\alpha\bigr)^{-1} \sqrt{\text{\rm Im}\, N(\lambda+i0)}.
 \end{equation*}
 \end{enumerate}
\end{thm}

\section{The operator \boldmath{$B_\lambda$} and the generalized trace}\label{sec3}

In this section we discuss properties of the operator $B_\lambda$ in~\eqref{blambdadef} and of the generalized trace defined 
in~\eqref{traceudef}. We verify that the latter is well-defined and independent of $\lambda$. Our investigation of the operator $B_\lambda$
is split into two parts: first the special case of a circle $\Sigma$ is treated, and afterwards the results are extended 
by perturbation arguments to the general case.

\subsection{Properties of \boldmath{$B_\lambda$} for a circle}

Throughout this subsection we assume that $\Sigma$ is a circle of radius $R = \frac{L}{2 \pi}$. Without loss of generality we assume that $\Sigma$ lies in the $xy$-plane and 
is centered at the origin. We will make use of its arc length parametrization  
\begin{align*}
 \sigma : [0,L] \to \R^3, \quad \sigma (t) = \big(R\cos(2\pi t/L),R\sin(2\pi t/L),0\big)
\end{align*}
and occasionally use the formula
\begin{align}\label{eq:chord}
 |\sigma (s) - \sigma (t)| = 2 R \sin\li  |s-t|\frac{\pi}{L} \re , \qquad s,t\in[0,L],
\end{align}
which holds for elementary geometric reasons. Furthermore, for $x = \sigma (t) \in \Sigma$ and $\delta > 0$ let $I_\delta^\Sigma (x)$ be the 
open interval in $\Sigma$ with center $x$ and length $2 \delta$ as in~\eqref{eq:IDeltaSigma}.

Let us first prove the following preliminary lemma. Its proof is partly inspired by~\cite[Lemma 1]{T90}.

\begin{lem}\label{Lemma_klambda}
Let $\lambda \leq 0$ and $x \in \Sigma$. Then the limit
\begin{align*}
 k_\lambda := \lim_{\delta \searrow 0} \bigg[ \int_{\Sigma \setminus I_\delta^\Sigma (x)} \frac{ e^{-\sqrt{-\lambda}|x-y|} }{4\pi|x-y|} \; d \sigma (y) +  \frac{\ln \delta }{2\pi} \bigg]
\end{align*}
exists in $\R$, is independent of $x$ and equals
\begin{align*}
 k_\lambda = \int_0^{\frac{\pi}{2}} \frac{ e^{-\sqrt{-\lambda}\cdot 2R\sin(s)}-1 }{ 2\pi\sin(s) } \; d s +  \frac{ \ln(4R) }{2\pi}.
\end{align*}
In particular, $k_\lambda \to - \infty$ as $\lambda \to - \infty$.
\end{lem}

\begin{proof}
First of all, it follows from the symmetry of the circle $\Sigma$ that $k_\lambda$ is
indeed independent of $x$ (if it exists). Hence, without loss of generality, we can choose $x = \sigma (0)$. Using \eqref{eq:chord} and the substitution $s = \frac{\pi}{L}t$ we obtain
\begin{align*}
 \int_{\Sigma \setminus I_\delta^\Sigma (x)} \frac{ e^{-\sqrt{-\lambda}|x-y|} }{4\pi|x-y|} \; d \sigma(y) & = \int_\delta^{L-\delta} \frac{ e^{-\sqrt{-\lambda}\cdot 2R\sin(\frac{\pi}{L} t)} }{4\pi\cdot 2R\sin(\frac{\pi}{L} t)} \; d t \\
 & = \int_{ \frac{\pi}{L}\delta }^{ \pi-\frac{\pi}{L}\delta } \frac{ e^{-\sqrt{-\lambda}\cdot 2R\sin(s)} }{ 4\pi \sin(s) } \; d s,
\end{align*}
where we have used $\frac{\pi}{L}=\frac{1}{2R}$ in the last equality. As $\sin(\frac{\pi}{2} - s) = \sin (\frac{\pi}{2} + s)$ for all $s \in \R$ it follows
\begin{align}\label{eq:calc_klambda}
\begin{split}
&\int_{\Sigma \setminus I_\delta^\Sigma (x)}  \frac{ e^{-\sqrt{-\lambda}|x-y|} }{4\pi|x-y|} \; d \sigma(y) + \frac{\ln \delta }{2\pi} \\
 &\quad = \int_{ \frac{\delta}{2R} }^{ \frac{\pi}{2} } \frac{ e^{-\sqrt{-\lambda}\cdot 2R\sin(s)} }{ 2\pi\sin(s) } \; d s  +  \frac{ \ln(\frac{\delta}{2R}) - \ln(\frac{\pi}{2}) + \ln(\pi R)  }{2\pi}  \\
 & \quad= \frac{1}{2\pi} \bigg[ \; \int_{ \frac{\delta}{2R} }^{ \frac{\pi}{2} } \frac{ e^{-\sqrt{-\lambda}\cdot 2R\sin(s)} }{ \sin(s) } \; d s   - \int_{ \frac{\delta}{2R} }^{ \frac{\pi}{2} } \frac{ 1 }{s} \; d s + \ln(\pi R) \bigg]\\
 & \quad= \frac{1}{2\pi} \bigg[ \; \int_{ \frac{\delta}{2R} }^{ \frac{\pi}{2} } \frac{ e^{-\sqrt{-\lambda}\cdot 2R\sin(s)}-1 }{ \sin(s) } \; d s  + \int_{ \frac{\delta}{2R} }^{ \frac{\pi}{2} } \frac{1}{\sin(s)}-\frac{1}{s} \; d s + \ln(\pi R) \bigg] .
\end{split}
\end{align}
With $\frac{d}{ds} \big( \ln(\sin(s/2)) - \ln(\cos(s/2)) \big) = \frac{1}{\sin s}$, $s\in(0,\frac{\pi}{2})$, we get 
$$\int_0^{ \frac{\pi}{2} } \left(\frac{1}{\sin(s)}-\frac{1}{s}\right) \; ds = \ln \left( \frac{4}{\pi}\right).$$ 
Hence in the limit $\delta \searrow 0$ the equation \eqref{eq:calc_klambda} becomes
\begin{align*}
 k_\lambda = \int_{0}^{ \frac{\pi}{2} } \frac{ e^{-\sqrt{-\lambda}\cdot 2R\sin(s)}-1 }{ 2\pi\sin(s) } \; d s +  \frac{ \ln(4R) }{2\pi}.
\end{align*}
In particular, $k_\lambda$ exists and is finite. By monotone convergence we have
\begin{align*}
 \int_{0}^{\frac{\pi}{2}} \frac{ 1- e^{- \sqrt{-\lambda}\cdot2R \sin(s)}  }{ \sin(s)} \; d s \to \int_{0}^{\frac{\pi}{2}} \frac{ 1 }{ \sin(s)} \; d s \geq 
 \int_{0}^{\frac{\pi}{2}} \frac{1}{s} \; d s = + \infty
\end{align*}
as $\lambda\to -\infty$, and hence $k_\lambda \to -\infty$ as $\lambda \to - \infty$.
\end{proof}

As a first step towards the study of the operator $B_\lambda$ on the circle we show properties of $B_0$ in the following lemma.

\begin{lem}\label{Lemma_T_0}
Consider the operator $B_0$ in~\eqref{blambdadef}, i.e., 
\begin{align*}
 (B_0 h) (x) & = \lim_{\delta \searrow 0} \bigg[ \int_{\Sigma \setminus I_\delta^\Sigma (x)} h(y) \frac{1}{4\pi|x-y|} \; d \sigma(y) + h(x) \frac{\ln \delta }{2\pi} \bigg], \quad h \in C^{0, 1} (\Sigma).
\end{align*}
Then the following assertions hold.
\begin{enumerate}
\item $B_0$ is a well-defined, essentially selfadjoint operator in $L^2(\Sigma)$.
\item $\overline{B_0}$ is bounded from above, has a compact resolvent, and its eigenvalues $\nu_k (0)$, $k = 1, 2, \dots$, ordered nonincreasingly and counted with multiplicities, are given by
\begin{align*}
 \nu_1 (0) = \frac{\ln(4R)}{2\pi}, \quad \nu_{2k} (0) = \nu_{2k+1} (0) = \frac{\ln(4R)}{2\pi} - \frac{1}{\pi} \sum_{j=1}^k  \frac{1}{2j-1}.
\end{align*}
\end{enumerate}
\end{lem}

\begin{proof}
Let $h\in C^{0,1}(\Sigma)$. For every $x \in \Sigma$ we can write
\begin{align*}
\begin{split}
 (B_0 h) (x) &= \int_\Sigma \frac{h(y)-h(x)}{4\pi|x-y|} d \sigma(y) \\
 &\qquad\quad + h(x) \lim_{\delta\searrow 0} \bigg[ \int_{\Sigma \setminus I_\delta^\Sigma (x)} \frac{1}{4\pi|x-y|} d \sigma(y) + \frac{\ln \delta }{2\pi} \bigg] .
\end{split}
 \end{align*}
Note that the first integral exists due to the fact that $h$ is Lipschitz continuous.
According to Lemma~\ref{Lemma_klambda} (for $\lambda=0$) we can write the above equation as
\begin{align}\label{eq:DarstellungT_0}
 (B_0 h) (x) = \int_\Sigma \frac{h(y)-h(x)}{4\pi|x-y|} d \sigma(y) + h(x) \frac{\ln(4R)}{2\pi},
\end{align}
where we have used $k_0 = \frac{\ln(4R)}{2\pi}$. It follows directly
\begin{align*}
 \left| (B_0 h) (x) \right| \leq \frac{R}{2} L_h + \frac{\ln(4R)}{2\pi}\|h\|_\infty,
\end{align*}
where $L_h$ is a Lipschitz constant of $h$. Thus $B_0$ is a well-defined operator in $L^2(\Sigma)$.

To show the symmetry of $B_0$ let $g, h\in C^{0,1}(\Sigma)$ be arbitrary. Using \eqref{eq:DarstellungT_0} we get
\begin{align*}
 \langle B_0 h, g \rangle_{L^2 (\Sigma)} & - \langle h, B_0 g \rangle_{L^2 (\Sigma)} \\
 & = \Big\langle  \Big[B_0-\frac{\ln(4R)}{2\pi} \Big]h,g \Big\rangle_{L^2 (\Sigma)} - \Big\langle h,\Big[B_0-\frac{\ln(4R)}{2\pi}\Big]  g \Big\rangle_{L^2 (\Sigma)}  \\
 & = \int_\Sigma \bigg( \int_\Sigma \frac{h(y)-h(x)}{4\pi|x-y|} \; d \sigma(y)  \bigg)   \overline{g(x)} d \sigma(x) \\
 & \quad - \int_\Sigma h(y) \overline{ \bigg( \int_\Sigma \frac{g(x)-g(y)}{4\pi|x-y|} \; d \sigma(x) \bigg) } d \sigma(y)\\
 & = \int_\Sigma \int_\Sigma \frac{h(y)\overline{g(y)}-h(x)\overline{g(x)}}{4\pi|x-y|} \; d \sigma(y) d \sigma(x) = 0,
\end{align*}
where the last equality follows from the fact that the integrand is skew-symmetric with respect to $x, y$. Thus $B_0$ is symmetric. 

Next we calculate the eigenvalues of $B_0$; this will also lead us to the essential 
selfadjointness of $B_0$. Consider the functions $h_k$ defined by $h_k (x) = \sin(kt/R)$ with $x=\sigma(t)$ and $k\in\N$. Then by~\eqref{eq:DarstellungT_0} and~\eqref{eq:chord} we have
\begin{align*}
\begin{split}
 \Big( \Big[ B_0 - \frac{\ln(4R)}{2\pi} \Big] h_k \Big)(x) &= \int_\Sigma  \frac{h_k(y)-h_k(x)}{4\pi|x-y|} d \sigma(y)\\
 &= \int_0^L \frac{ \sin(ks/R)-\sin(kt/R)}{4\pi\cdot 2 R \sin\li \frac{|s-t|}{2R}\re} d s.
\end{split}
 \end{align*}
Due to the identity $\sin(ks/R)-\sin(kt/R) = 2 \sin ( \frac{ks-kt}{2R} ) \cos ( \frac{ks+kt}{2R} )$ this leads to
\begin{align}\label{calc1}
\begin{split}
 \Big( \Big[ B_0 - \frac{\ln(4R)}{2\pi} \Big] h_k \Big)(x) & = \int_{0}^{L} \frac{  \sin \big( \frac{k(s-t)}{2R} \big) \cos \big( \frac{k(s+t)}{2R}\big) }{4\pi R \sin \big( \frac{|s-t|}{2R} \big) } d s . 
\end{split}
\end{align}
We split the interval of integration into two parts and obtain with the substitution $z = s - t + L$ for the first integral
\begin{align}\label{calc1_1}
\begin{split}
 &\int_{0}^{t} \frac{\sin \big( \frac{k(s-t)}{2R} \big) \cos \big( \frac{k(s+t)}{2R} \big) }{4\pi R \sin \big( \frac{t-s}{2R} \big) } d s \\
 & \qquad\qquad = \int_{L-t}^{L} \frac{  \sin \big( \frac{k(z-L)}{2R} \big) \cos \big( \frac{k(z-L+2t)}{2R} \big) }{4\pi R \sin \big( \frac{L-z}{2R} \big) } d z  \\
 & \qquad\qquad = \int_{L-t}^{L} \frac{  \sin\li \frac{kz}{2R}-k\pi \re \cos\li \frac{kz}{2R}-k\pi+\frac{kt}{R} \re   }{4\pi R \sin\li \pi-\frac{z}{2R} \re} d z \\
 & \qquad\qquad = \int_{L-t}^{L} \frac{  \sin\li \frac{kz}{2R} \re \cos\li \frac{kz}{2R}+\frac{kt}{R} \re   }{4\pi R \sin\li \frac{z}{2R} \re} d z  ,
\end{split}
\end{align}
where we have used in the last step that $\sin$ is an odd function and that the formulas $\sin(x+\pi) = -\sin(x)$ and $\cos(x+\pi) = -\cos(x)$ hold for 
all $x \in \R$. For the remaining second integral the substitution $z = s - t$ yields
\begin{align}\label{calc1_2}
 \int_t^L \frac{  \sin \big( \frac{k(s-t)}{2R} \big) \cos \big( \frac{k(s+t)}{2R} \big) }{4\pi R \sin \big( \frac{s-t}{2R} \big) } d s = \int_0^{L-t} \frac{  \sin\li \frac{kz}{2R}\re \cos\li \frac{kz}{2R} + \frac{kt}{R}\re   }{4\pi R \sin\li \frac{z}{2R}\re} d z .
\end{align}
With the help of~\eqref{calc1_1} and~\eqref{calc1_2} and the substitution $s = z/(2 R)$ the identity~\eqref{calc1} implies 
\begin{align}\label{calc2}
\begin{split}
 & \Big( \Big[ B_0 - \frac{\ln(4R)}{2\pi} \Big] h_k \Big) (x) \\
 & \qquad\quad = \int_{0}^{L} \frac{  \sin\li \frac{kz}{2R} \re \cos\li \frac{kz}{2R}+\frac{kt}{R} \re   }{4\pi R \sin\li \frac{z}{2R} \re} \; d z \\
 & \qquad\quad = \int_{0}^{\pi} \frac{  \sin(ks) \cos\li ks + \frac{kt}{R}\re }{2\pi\sin(s)} \; d s \\
 & \qquad\quad = \int_{0}^{\pi} \frac{ \sin(ks) }{ 2\pi\sin(s) } \Big[ \cos(ks) \cos \Big(\frac{kt}{R} \Big)  - \sin(ks)  \sin \Big( \frac{kt}{R} \Big) \Big] d s \\
 & \qquad\quad = - \sin \Big( \frac{kt}{R} \Big) \int_0^\pi \frac{  \sin^2(ks)  }{2\pi\sin(s)} d s,
\end{split}
\end{align}
where
\begin{equation*}
 \int_{0}^{\pi} \frac{ \sin(ks) \cos(ks)}{ 2\pi\sin(s) }  d s =0
\end{equation*}
was used in the last step. Furthermore, using the identity $2\sin^2(ks) = 1 - \cos(2ks)$ and the indefinite integrals given in~\cite[2.526~1. and 2.539~4.]{GR} we get
\begin{align*}
 \int_{0}^{\pi} \frac{  \sin^2(ks)  }{2\pi \sin(s)} \; d s & = \frac{1}{4\pi}  \int_{0}^{\pi} \frac{1}{\sin(s)}  - \frac{ \cos(2ks) }{\sin(s)} \; d s \\
 & = - \frac{1}{2\pi} \sum_{j=1}^k  \frac{ \cos[(2j-1)s] }{2j-1} \bigg|_{0}^\pi \\
 & = \frac{1}{\pi} \sum_{j=1}^k \frac{1}{2j-1}.
\end{align*}
Hence~\eqref{calc2} yields
\begin{align}\label{eq:EF1}
 \Big( \Big[ B_0 - \frac{\ln(4R)}{2\pi} \Big] h_k \Big) (x) = - \bigg( \frac{1}{\pi} \sum_{j=1}^k  \frac{1}{2j-1} \bigg) h_k(x).
\end{align}
By an analogous computation we see that also
\begin{align}\label{eq:EF2}
 \Big( \Big[ B_0 - \frac{\ln(4R)}{2\pi} \Big] \widetilde h_k \Big) (x) = - \bigg( \frac{1}{\pi} \sum_{j=1}^k  \frac{1}{2j-1} \bigg) \widetilde h_k(x),
\end{align}
where $\widetilde h_k (x) = \cos(kt/R)$ with $x=\sigma(t)$. Moreover, for the constant function $h (x) = 1$ on $\Sigma$ we clearly have 
\begin{align}\label{eq:EF3}
 \Big[B_0 - \frac{\ln(4R)}{2 \pi} \Big] h = 0.
\end{align}
Since the functions $h, h_k, \widetilde h_k$ are eigenfunctions of $B_0$ by~\eqref{eq:EF1},~\eqref{eq:EF2} and~\eqref{eq:EF3} and span a dense 
subspace of $L^2 (\Sigma)$, it follows that the symmetric operator $B_0$ is actually essentially selfadjoint in $L^2 (\Sigma)$. 
Furthermore, by~\eqref{eq:EF1},~\eqref{eq:EF2} and~\eqref{eq:EF3}, the selfadjoint closure $\overline{B_0}$ has a pure point spectrum and its eigenvalues, 
counted with multiplicities, are given by $\nu_k(0)$, $k=1,2,\dots$, in item~(ii). Since these eigenvalues are bounded from above and converge to 
$- \infty$ as $k \to + \infty$, it follows that $\overline{B_0}$ is bounded from above and has a compact resolvent. 
\end{proof}

Let us now turn to the operator $B_\lambda$ on the circle for general $\lambda < 0$.
\begin{lem}\label{Lemma_T_lambda}
Let $\lambda \leq 0$, let $\Sigma$ be a circle of radius $R$ and let 
$B_\lambda$ be defined in~\eqref{blambdadef}.
Then the following assertions hold.
\begin{enumerate}
  \item $B_\lambda$ is a well-defined, essentially selfadjoint operator in $L^2 (\Sigma)$ and the identity $\dom \overline{B_\lambda} = \dom \overline{B_0}$ holds.
  \item $\overline{B_\lambda}$ is bounded from above and has a compact 
resolvent.
  \item The eigenvalues $\nu_k (\lambda)$ of $\overline{B_\lambda}$, $k 
= 1, 2, \dots$, ordered nonincreasingly and counted with multiplicities, 
satisfy
  \begin{align*}
   \nu_k (\lambda) = - \frac{\ln k}{2\pi} + O(1) \quad \text{as} \quad k 
\to + \infty  .
  \end{align*}
  \item The largest eigenvalue $\nu_1(\lambda)$ of $\overline{B_\lambda}$ is given by
$k_\lambda$ in Lemma~\ref{Lemma_klambda}. In 
particular, $\nu_k (\lambda) \to - \infty$ as $\lambda \to - \infty$, $k 
= 1, 2, \dots$. The eigenspace corresponding to $\nu_1(\lambda)$ is given by the constant functions on $\Sigma$.
\end{enumerate}
\end{lem}

\begin{proof}
Note first that the operator $B_\lambda$ can be written as
\begin{align}\label{eq:BlambdaDecomp}
  B_\lambda = B_0 - M_\lambda,
\end{align}
where 
\begin{align*}
  (M_\lambda h)(x) = \int_\Sigma h(y) 
\frac{1-e^{-\sqrt{-\lambda}|x-y|}}{4\pi|x-y|} d \sigma(y), \quad h \in 
L^2 (\Sigma).
\end{align*}
The integral operator $M_\lambda$ has a real, symmetric kernel, which is square integrable since for all $x, y \in \Sigma$ there 
exists $\xi \in [ -\sqrt{-\lambda}|x-y|,0 ]$ with
\begin{align*}
  \bigg|   \frac{1-e^{-\sqrt{-\lambda}|x-y|}}{4\pi|x-y|} \bigg| = \frac{\big| e^0-e^{-\sqrt{-\lambda}|x-y|} \big|}{4\pi|x-y|} = \frac{ e^\xi \big| 0 - (- \sqrt{-\lambda}|x-y| ) \big|}{4\pi|x-y|} \leq \frac{ 
\sqrt{-\lambda} }{4\pi}.
\end{align*}
Thus $M_\lambda$ is a compact, selfadjoint operator in $L^2 (\Sigma)$.
Hence, due to Lemma~\ref{Lemma_T_0} 
and~\eqref{eq:BlambdaDecomp} $B_\lambda$ is well-defined and essentially 
selfadjoint in $L^2 (\Sigma)$ with
\begin{align}\label{eq:BlambdaClosureDecomp}
  \overline{B_\lambda} = \overline{B_0} - M_\lambda.
\end{align}
In particular, $\overline{B_\lambda}$ has a compact resolvent and $\dom 
\overline{B_\lambda} = \dom \overline{B_0}$, which shows~(i).

Next we show that $\overline{B_\lambda}$ is bounded from above by the 
number $k_\lambda$ defined in Lemma~\ref{Lemma_klambda}. For every $h 
\in C^{0,1}(\Sigma)$ and $x \in \Sigma$ we can write
\begin{align*}
(B_\lambda h)(x) = \int_\Sigma \big[ h(y)-h(x) \big] \frac{ 
e^{-\sqrt{-\lambda}|x-y|} }{4\pi|x-y|} d \sigma(y)  +
k_\lambda \cdot h(x),
\end{align*}
where again the integral exists due to the Lipschitz continuity of $h$. Hence
\begin{align*}
  \bigl\langle (B_\lambda - k_\lambda) h,h \bigr\rangle_{L^2(\Sigma)} & = \int_\Sigma \bigg( \int_\Sigma \big[ h(y)-h(x) \big] \frac{ e^{-\sqrt{-\lambda}|x-y|} }{4\pi|x-y|} \; d \sigma(y) \bigg) \overline{h(x)} d \sigma(x)\\
  & = \int_\Sigma \int_\Sigma \big[ h(y) - h(x) \big] \frac{ e^{-\sqrt{-\lambda}|x-y|} }{4\pi|x-y|} \overline{h(x)} d \sigma(y) d \sigma(x)\\
  & = - \int_\Sigma \int_\Sigma \big[ h(y)-h(x) \big] \frac{ e^{-\sqrt{-\lambda}|x-y|} }{4\pi|x-y|} \overline{h(y)} d \sigma(y) d \sigma(x),
\end{align*}
where in the last step we first changed the roles of $x$ and $y$ and 
then the order of integration. Addition of the last two lines yields
\begin{align*}
\begin{split}
  &2 \bigl\langle (B_\lambda-k_\lambda) h,h \bigr\rangle_{L^2(\Sigma)} \\
  &\qquad = \int_\Sigma \int_\Sigma \big[ 
h(y)-h(x) \big] \frac{ e^{-\sqrt{-\lambda}|x-y|} }{4\pi|x-y|} \left[ 
\overline{h(x)} - \overline{h(y)} \right] d \sigma(y) d \sigma(x) \\ 
& \qquad\leq 0
\end{split}
\end{align*}
and, hence, $\langle B_\lambda h,h\rangle_{L^2(\Sigma)} \leq k_\lambda \langle h,h\rangle_{L^2(\Sigma)}$ for all 
$h\in C^{0,1}(\Sigma)$, with equality if and only if $h$ is constant, that is, 
$B_\lambda$ (and, thus, $\overline{B_\lambda}$) is bounded from above by 
$k_\lambda$, which shows~(ii). Moreover it follows $\nu_1 (\lambda) = 
k_\lambda$. By Lemma~\ref{Lemma_klambda} this implies $\nu_1 (\lambda) 
\to - \infty$ as $\lambda \to - \infty$ and thus $\nu_k (\lambda) \to - 
\infty$ as $\lambda \to - \infty$ for all $k$. This finishes the proof 
of~(iv).

It remains to verify the asymptotic behaviour of the eigenvalue $\nu_k 
(\lambda)$ for $k \to + \infty$ as claimed in~(iii). According 
to~\cite[Equation 4.1.32]{AS} we have
\begin{align*}
  \sum_{j=1}^k \frac{1}{j} = \ln(k) + \gamma + o(1) \quad \text{as} \quad k \to + \infty,
\end{align*}
where $\gamma \approx 0.577216$ denotes the Euler--Mascheroni constant. Hence
\begin{align*}
\begin{split}
\sum_{j=1}^k \frac{1}{2j-1}
& = \sum_{j=1}^{2 k} \frac{1}{j} - \frac{1}{2} \sum_{j=1}^k \frac{1}{j}
  = \ln(2k) + \gamma - \frac{\ln(k)  + \gamma}{2} + o(1)\\
& = \frac{\gamma}{2} + \frac{\ln(4k)}{2} + o(1)
\quad \text{as} \quad  k \to +\infty .
\end{split}
\end{align*}
By Lemma~\ref{Lemma_T_0}~(ii) for the eigenvalues of
$\overline{B_0}$ this implies
\begin{align}\label{eq:asympKreis}
\begin{split}
  \nu_{2k} (0) & = \frac{\ln(4R)}{2\pi} - \frac{1}{\pi} \sum_{j=1}^k  
\frac{1}{2j-1}
  = \frac{\ln(4R)}{2\pi} - \frac{\gamma}{2\pi} - \frac{\ln(4k)}{2\pi} + 
o(1)  \\
  & = -\frac{\ln k}{2\pi} + \frac{\ln R - \gamma}{2\pi} + o(1)
= -\frac{\ln(2k)}{2\pi} + O(1)
 \quad 
\text{as} \quad k \to + \infty
\end{split}
\end{align}
and consequently
\begin{align}\label{eq:asympKreis2}
\begin{split}
\nu_{2k+1}(0) 
= \nu_{2k}(0)
&=  -\frac{\ln(2k+1)-\ln(\frac{2k+1}{2k} )}{2\pi} + O(1) \\
&=  -\frac{\ln(2k+1)}{2\pi} + O(1)    \quad \text{as} \quad  k \to +\infty .
\end{split}
\end{align}
From~\eqref{eq:BlambdaClosureDecomp} we conclude with the help of the 
min-max principle
\begin{align*}
  \nu_k (0) - \|M_\lambda\| \leq \nu_k (\lambda) \leq \nu_k (0) + 
\|M_\lambda\|, \quad k = 1, 2, \dots.
\end{align*}
The latter together with~\eqref{eq:asympKreis} and~\eqref{eq:asympKreis2} implies
\begin{align*}
  \nu_k (\lambda) = \nu_k (0) + O(1) = -\frac{\ln k}{2\pi} + O(1) \quad 
\text{as} \quad k\rightarrow +\infty,
\end{align*}
which completes the proof of the lemma.
\end{proof}

\subsection{Properties of \boldmath{$B_\lambda$} in the general case}

In this subsection $\Sigma$ is an arbitrary compact, closed, regular $C^2$-curve in $\R^3$ of length $L$ without self-intersections. In the following we explore properties of $B_\lambda$ by using the results of the previous subsection for the case of a circle. This will be done by a perturbation argument. 

Let $\T$ be a circle in $\R^3$ with radius $R = \frac{L}{2 \pi}$ which is parametrized with respect to the arc length by a function $\tau : [0, L] \to \R^3$. In order to distinguish the operators $B_\lambda$ on $\Sigma$ from those on the circle $\T$ we denote the latter by $B_\lambda^\T$. Moreover, recall that $\sigma : [0, L] \to \R^3$ is an arc length parametrization of $\Sigma$. We define an operator $D_\lambda$ by
\begin{align}\label{Def_Dlambda}
 \big( D_\lambda h \big)( \sigma(t)) = \int_0^L h (\sigma(s)) \bigg[ \frac{e^{-\sqrt{-\lambda}|\sigma(t)-\sigma(s)|}}{4\pi|\sigma(t)-\sigma(s)|} - \frac{e^{-\sqrt{-\lambda}|\tau(t)-\tau(s)|}}{4\pi|\tau(t)-\tau(s)|} \bigg] d s
\end{align}
for $h \in L^2 (\Sigma)$. Furthermore, let $J : L^2(\Sigma) \to L^2(\mathcal{T})$ be the unitary operator defined by
\begin{align}\label{eq:J}
 J h = h\circ\sigma\circ\tau^{-1}, \quad h \in L^2 (\Sigma).
\end{align}
Our studies of $B_\lambda$ will rely on the following properties of $D_\lambda$ .

\begin{lem}\label{lem:Dlambda}
For each $\lambda \leq 0$ the operator $D_\lambda$ in~\eqref{Def_Dlambda} is well-defined, compact and selfadjoint in $L^2 (\Sigma)$, 
and $\| D_\lambda \|\leq C$ holds for all $\lambda \leq 0$ and some $C > 0$
which is independent of $\lambda$. In the special case $\lambda = 0$ the estimate
\begin{align}\label{eq:D0est}
 \|D_0\| \leq d_\Sigma
\end{align}
holds with $d_\Sigma$ given in~\eqref{ddsigma}. Moreover, the relation
\begin{align}\label{eq:perturbation}
 B_\lambda = D_\lambda + J^* B_\lambda^\T J
\end{align}
is satisfied for all $\lambda \leq 0$.
\end{lem}

\begin{proof}
In order to study the integral in the definition~\eqref{Def_Dlambda} of $D_\lambda$ we identify the parametrizations $\sigma, \tau$ of $\Sigma$ and $\T$, respectively, with their $L$-periodic continuations to all of $\R$. Let $s, t \in\R$ with $|s-t|\leq \frac{L}{2}$. Define $f:(0,\infty)\to\R$ via $f(z) = \frac{e^{-\sqrt{-\lambda}z}}{4\pi z}$ for $z > 0$. Then
\begin{align}\label{eq:f'}
 f'(z) = \frac{ -\sqrt{-\lambda}e^{-\sqrt{-\lambda}z}4\pi z    - e^{-\sqrt{-\lambda}z} 4\pi}{ (4\pi z)^2} = - e^{-\sqrt{-\lambda}z}  \frac{ \sqrt{-\lambda} z    +   1 }{ 4\pi z^2},
\end{align}
from which it follows that $f'$ is monotonously nondecreasing on $(0, \infty)$ and, thus, $|f'|$ is monotonously nonincreasing on $(0, \infty)$. Hence with 
$$\zeta_{\min} = \min\big\{|\sigma (t) - \sigma (s)|,|\tau (t) - \tau (s)| \big\}$$ it follows
\begin{align}\label{estimate1}
\begin{split}
& \bigg| \frac{e^{-\sqrt{-\lambda}|{\sigma}(t)-{\sigma}(s)|}}{4\pi|{\sigma}(t)-{\sigma}(s)|} - \frac{e^{-\sqrt{-\lambda}|{\tau}(t)-{\tau}(s)|}}{4\pi|{\tau}(t)-{\tau}(s)|}  \bigg|\\
 &\qquad\qquad\qquad\qquad  \leq |f'(\zeta_{\min})| \cdot \big|  |{\sigma}(t)-{\sigma}(s)|   - |{\tau}(t)-{\tau}(s)|    \big| .
\end{split}
 \end{align}
Note, that there exist $\varepsilon_\sigma>0$ and $\varepsilon_\tau>0$ such that for all $s,t\in\R$ with $|s-t|\leq \frac{L}{2}$
\begin{align*}
 |{\sigma}(s)-{\sigma}(t)| \geq \varepsilon_\sigma |s-t| \quad \text{and} \quad |{\tau}(s)-{\tau}(t)| \geq \varepsilon_\tau |s-t|
\end{align*}
holds. With $\varepsilon:=\min\{\varepsilon_\sigma,\varepsilon_\tau\}>0$ the estimate \eqref{estimate1} can be simplified to
\begin{align}\label{estimate2}
\begin{split}
&\bigg| \frac{e^{-\sqrt{-\lambda}|{\sigma}(t)-{\sigma}(s)|}}{4\pi|{\sigma}(t)-{\sigma}(s)|} - \frac{e^{-\sqrt{-\lambda}|{\tau}(t)-{\tau}(s)|}}{4\pi|{\tau}(t)-{\tau}(s)|}  \bigg| \\
 &\qquad\qquad\qquad\qquad 
  \leq |f'(\varepsilon|s-t|)|  \big|  |{\sigma}(t)-{\sigma}(s)|   - |{\tau}(t)-{\tau}(s)|    \big| .
\end{split}
  \end{align}
Recall that $\Sigma$ is a $C^2$-curve. Hence we get with Taylor's theorem (for each component) for some suitable $\zeta_1$, $\zeta_2$ and $\zeta_3$
\begin{align*}
{\sigma}(t)
&=
\begin{bmatrix}
{\sigma}_1(t)  \\
{\sigma}_2(t)  \\
{\sigma}_3(t)
\end{bmatrix}
=
{\sigma}(s)  +  {\sigma}'(s)(t-s)  +
\begin{bmatrix}
{\sigma}_1''(\zeta_1)\\
{\sigma}_2''(\zeta_2)\\
{\sigma}_3''(\zeta_3)
\end{bmatrix}
\frac{(t-s)^2}{2}.
\end{align*}
With $C_\sigma := \sqrt{ \|{\sigma}_1''\|_\infty^2 + \|{\sigma}_2''\|_\infty^2 + \|{\sigma}_3''\|_\infty^2 }$ and $|{\sigma}'(s)|=1$ it follows
\begin{align*}
|{\sigma}(t)-{\sigma}(s)|
&\leq
|{\sigma}'(s)| \cdot |t-s|  + \left|
\begin{bmatrix}
{\sigma}_1''(\zeta_1)\\
{\sigma}_2''(\zeta_2)\\
{\sigma}_3''(\zeta_3)
\end{bmatrix}\right| \frac{(t-s)^2}{2}
\leq
|t-s| + \frac{C_\sigma}{2} |t-s|^2.
\end{align*}
Analogously we get with $C_\tau := \sqrt{ \|{\tau}_1''\|_\infty^2 + \|{\tau}_2''\|_\infty^2 + \|{\tau}_3''\|_\infty^2 }$
\begin{align*}
|{\tau}(t)-{\tau}(s)|
&\geq
|{\tau}'(s)| \cdot |t-s|  - \left|
\begin{bmatrix}
{\tau}_1''(\xi_1)\\
{\tau}_2''(\xi_2)\\
{\tau}_3''(\xi_3)
\end{bmatrix}\right| \frac{(t-s)^2}{2}
\geq
|t-s| - \frac{C_\tau}{2} |t-s|^2
\end{align*}
for some suitable $\xi_1$, $\xi_2$ and $\xi_3$. Hence
\begin{align*}
|{\sigma}(t)-{\sigma}(s)|   - |{\tau}(t)-{\tau}(s)|
&\leq
\frac{C_\sigma+C_\tau}{2} |t-s|^2  .
\end{align*}
By changing the roles of ${\sigma}$ and ${\tau}$ we observe
\begin{align}\label{estimate3}
\Big| |{\sigma}(t)-{\sigma}(s)|   - |{\tau}(t)-{\tau}(s)| \Big|
\leq
\frac{C_\sigma+C_\tau}{2} |t-s|^2  .
\end{align}
Note that $e^{-x}(x+1)\leq 1$ for $x\geq0$. Together with \eqref{eq:f'}, \eqref{estimate3} and 
$$\widetilde{C}:= \frac{C_\sigma+C_\tau}{8\pi\varepsilon^2}$$ 
the estimate 
\eqref{estimate2} implies
\begin{align}\label{estimate4}
\begin{split}
 &\bigg| \frac{e^{-\sqrt{-\lambda}|{\sigma}(t)-{\sigma}(s)|}}{4\pi|{\sigma}(t)-{\sigma}(s)|} - \frac{e^{-\sqrt{-\lambda}|{\tau}(t)-{\tau}(s)|}}{4\pi|{\tau}(t)-{\tau}(s)|}  \bigg| \\
 &\qquad\qquad\qquad\leq \widetilde{C} e^{-\sqrt{-\lambda}\varepsilon |s-t|}  \big[  \sqrt{-\lambda} \varepsilon |s-t| +  1  \big] \leq \widetilde{C}
 \end{split}
 \end{align}
for all $s,t\in\R$ with $|s-t|\leq \frac{L}{2}$. For arbitrary $s,t\in\R$ there exists $k\in\mathbbm{Z}$ such that $|(s+kL) -t| \leq \frac{L}{2}$. 
As ${\sigma}$ and ${\tau}$ are $L$-periodic it follows that~\eqref{estimate4} holds for all $s,t\in\R$. From~\eqref{estimate4} we conclude
that the integral kernel of the operator $D_\lambda$ is bounded with a bound~$\widetilde C$ independent of $\lambda$. Thus with $C = \widetilde{C} L$, the definition of $D_\lambda$ in \eqref{Def_Dlambda} and estimate \eqref{estimate4} it follows
\begin{align*}
 \| D_\lambda h \|_{L^2(\Sigma)}^2 & \leq \|h\|_{L^2(\Sigma)}^2 
 \int_0^L \int_0^L \bigg| \frac{e^{-\sqrt{-\lambda}|\sigma(t)-\sigma(s)|}}{4\pi|\sigma(t)-\sigma(s)|} - \frac{e^{-\sqrt{-\lambda}|\tau(t)-\tau(s)|}}{4\pi|\tau(t)-\tau(s)|}\bigg|^2 ds \; dt \\
 & \leq  C^2\|h\|_{L^2(\Sigma)}^2
\end{align*}
for all $h \in L^2(\Sigma)$ and $C$ does not depend on $\lambda$. In particular, $D_\lambda$ is a well-defined, 
compact operator in $L^2 (\Sigma)$ whose operator norm can be estimated by a constant independent of $\lambda$. Since the integral kernel of $D_\lambda$ is real and symmetric 
it follows that $D_\lambda$ is selfadjoint. For $\lambda = 0$ the estimate~\eqref{eq:D0est} follows immediately from the definition of $D_\lambda$. 

In order to verify the relation~\eqref{eq:perturbation} note that $h\in C^{0,1}(\Sigma)$ if and only if $\widetilde h := J h \in C^{0,1}(\mathcal{T})$ and in this case
\begin{align*}
 \big(J^* B_\lambda^\T J  h \big)(x) 
 = \lim_{\delta\searrow 0} \bigg[  \int_{\mathcal{T}\setminus I_\delta^\mathcal{T}(\tau(t))} \widetilde{h}(\tilde{y}) \frac{e^{-\sqrt{-\lambda}|\tau(t)-\widetilde{y}|}}{4\pi|\tau(t)-\widetilde{y}|} d \sigma(\widetilde{y}) + \widetilde{h}(\tau(t)) \frac{\ln \delta }{2\pi} \bigg]
\end{align*}
for every $h\in C^{0,1}(\Sigma)$ and $x=\sigma(t)\in\Sigma$. This identity and the definitions of $B_\lambda$ and $D_\lambda$ lead to the relation~\eqref{eq:perturbation}.
\end{proof}

Now we are in the position to prove all properties of $B_\lambda$ which are required for the proofs of the main results of this paper.

\begin{prop}\label{Proposition_Main}
Let $\lambda \leq 0$ and let $B_\lambda$ be given in~\eqref{blambdadef}. Then the following assertions hold.
\begin{enumerate}
 \item $B_\lambda$ is a well-defined, essentially selfadjoint operator in $L^2(\Sigma)$ and the identity $\dom \overline{B_\lambda} = \dom \overline{B_0}$ holds.
 \item $\overline{B_\lambda}$ is bounded from above and has a compact resolvent.
 \item The eigenvalues $\nu_k (\lambda)$ of $\overline{B_\lambda}$, $k = 1, 2, \dots$, ordered nonincreasingly and counted with multiplicities, satisfy
 \begin{align*}
  \nu_k (\lambda) = - \frac{\ln k}{2\pi} + O(1) \quad \text{as} \quad k \to + \infty.
 \end{align*}
 \item For every $k \in \N$ the function $\lambda \mapsto \nu_k (\lambda)$ is continuous and strictly increasing on the 
 interval $(-\infty,0]$ and $\nu_k (\lambda) \to -\infty$ as $\lambda \to-\infty$.
\end{enumerate}
\end{prop}

\begin{proof}
Let $D_\lambda$ be given in~\eqref{Def_Dlambda} and let $J : L^2 (\Sigma) \to L^2 (\T)$ be the unitary operator in~\eqref{eq:J}. Since $D_\lambda$ 
is selfadjoint and compact in $L^2 (\Sigma)$ by Lemma~\ref{lem:Dlambda}, the assertions in~(i) and~(ii) follow directly from~\eqref{eq:perturbation} and Lemma~\ref{Lemma_T_lambda}~(i) and~(ii). Furthermore, by~\eqref{eq:perturbation}, Lemma~\ref{Lemma_T_lambda}~(ii) and~(iv), and Lemma~\ref{lem:Dlambda} there exists 
$C > 0$ independent of $\lambda$ such that for $h \in \dom \overline{B_\lambda}$ we have
\begin{align}\label{eq:estimate99}
\begin{split}
\sk{\overline{B_\lambda}h,h}_{L^2(\Sigma)}
& = \sk{D_\lambda h,h}_{L^2(\Sigma)} + \bigl\langle\overline{ B_\lambda^\T }Jh,Jh\bigr\rangle_{L^2(\mathcal{T})}  \\
& \leq \|D_\lambda\|\cdot\|h\|_{L^2(\Sigma)}^2 + k_\lambda \|Jh\|_{L^2(\mathcal{T})}^2 \\
&\leq (C + k_\lambda) \|h\|_{L^2(\Sigma)}^2, 
\end{split}
\end{align}
where $k_\lambda$ is given in Lemma~\ref{Lemma_klambda}. Since $k_\lambda \to - \infty$ as $\lambda \to - \infty$ by Lemma~\ref{Lemma_klambda} we conclude from \eqref{eq:estimate99} that $\nu_k (\lambda) \to - \infty$ as $\lambda \to - \infty$ for each $k$. From~\eqref{eq:perturbation} and the min-max principle it follows
\[
 \nu_k(\lambda) - C \leq \nu_k^\T (\lambda) \leq \nu_k (\lambda) + C, \quad k = 1, 2, \dots,
\]
where $\nu_k^\T(\lambda)$ denotes the $k$-th eigenvalue of $\overline{ B_\lambda^\T }$.
We obtain with the help of Lemma~\ref{Lemma_T_lambda}~(iii) that
\begin{align*}
 \nu_k(\lambda) = \nu_k^\T (\lambda) + O (1) = -\frac{\ln k}{2\pi} + O(1) \quad \text{as} \quad k \rightarrow + \infty.
\end{align*}
This proves the assertion~(iii). 

In order to show the remaining assertions in~(iv) let $\lambda,\mu \leq 0$ and define the operator $D_{\lambda,\mu} : L^2(\Sigma) \to L^2(\Sigma)$ by
\begin{align*}
(D_{\lambda,\mu} h) (x) = \int_\Sigma  h(y)\frac{ e^{-\sqrt{-\lambda}|x-y|} - e^{-\sqrt{-\mu}|x-y|} }{4\pi|x-y|} \, d\sigma(y),
\quad h \in L^2 (\Sigma).
\end{align*}
As $B_\lambda h-B_\mu h=D_{\lambda,\mu} h$ for all $h\in C^{0,1}(\Sigma)$ it follows that
\begin{align}\label{eq:nochmalPerturbation}
 \overline{B_\lambda}h = \overline{B_\mu}h + D_{\lambda, \mu}h,\qquad h\in\dom\overline{B_\lambda}.
\end{align}
As in the proof of Lemma~\ref{Lemma_T_lambda} one shows that $D_{\lambda,\mu}$ is a compact, selfadjoint operator with
\begin{align*}
 \| D_{\lambda, \mu} \| \leq \frac{ |\sqrt{-\lambda}-\sqrt{-\mu}| }{4\pi} L.
\end{align*}
In particular, $\|D_{\lambda,\mu}\| \to 0$ as $\lambda \to \mu$. From this and~\eqref{eq:nochmalPerturbation} it follows with the min-max principle that $\nu_k (\lambda) \to \nu_k (\mu)$ for all $k$, that is, all the functions $\lambda\mapsto\nu_k(\lambda)$ are continuous.

For the strict monotonicity let $\lambda, \mu < 0$. If $h \in \dom B_\lambda = \dom B_\mu$ it follows from 
the definition of $\gamma_\lambda$ and $\gamma_\mu$ in~\eqref{gammaop} that
\begin{align}\label{eq:uc}
\begin{split}
 \gamma_\lambda h-\gamma_\mu h & =  (-\Delta-\lambda)^{-1}(h\delta_\Sigma) - (-\Delta-\mu)^{-1}(h\delta_\Sigma) \\
 & = (\lambda-\mu)(-\Delta-\lambda)^{-1}(-\Delta-\mu)^{-1}(h\delta_\Sigma),
\end{split}
\end{align}
in particular, $\gamma_\lambda h - \gamma_\mu h \in H^2 (\R^3)$. Note also that $\gamma_\lambda-\gamma_\mu$ is continuous from $L^2(\Sigma)$ to
$H^2(\R^3)$ since $\gamma_\lambda-\gamma_\mu$ is defined on $L^2(\Sigma)$ and is closed as a mapping from $L^2(\Sigma)$ to
$H^2(\R^3)$.
According to Lemma~\ref{Lemma_gammafield} we have
\begin{equation}\label{eq:integral}
  \begin{split}
(\gamma_\lambda h-\gamma_\mu h)(x) & = \int_\Sigma h(s)\frac{e^{-\sqrt{-\lambda}\vert x-s\vert}-e^{-\sqrt{-\mu}\vert x-s\vert}}{4\pi \vert x-s\vert} \, ds
\end{split}
\end{equation}
for almost all $x\in\R^3\setminus\Sigma$. As the integral in \eqref{eq:integral} is continuous with respect to $x$ we obtain~\eqref{eq:integral} for all $x\in\R^3$. In particular, 
\begin{equation}\label{eq:integral2}
 \begin{split}
  (\gamma_\lambda h-\gamma_\mu h)|_\Sigma(x)
&=
\int_\Sigma h(s)\frac{e^{-\sqrt{-\lambda}\vert 
x-s\vert}-e^{-\sqrt{-\mu}\vert x-s\vert}}{4\pi \vert x-s\vert} \, ds\\
&=
(B_\lambda h-B_\mu h)(x)
 \end{split}
\end{equation}
for all $x\in\Sigma$ and $h\in C^{0,1}(\Sigma)=\dom B_\lambda=\dom B_\mu$. If $h\in\dom\overline{B_\lambda}=\dom\overline{B_\mu}$
we can choose a sequence $(h_n)\subset \dom B_\lambda = \dom B_\mu$
with $h_n\rightarrow h$ and $B_\lambda h_n \rightarrow \overline{B_\lambda}h$.
Due to \eqref{eq:integral2} and \eqref{eq:uc} we observe
\begin{align*}
B_\lambda h_n
&=
B_\mu h_n + (\gamma_\lambda h_n-\gamma_\mu h_n)|_\Sigma \\
&=
B_\mu h_n + \big((\lambda-\mu)(-\Delta-\lambda)^{-1}(-\Delta-\mu)^{-1}(h_n\delta_\Sigma)\big)|_\Sigma.
\end{align*}
Since the mapping $h\mapsto h\delta_\Sigma$ is continuous from $L^2(\Sigma)$ to $H^{-2}(\R^3)$ (see \eqref{hdelta}),
 $-\Delta-\lambda$ is an isomorphism between $H^s(\R^3)$ and $H^{s-2}(\R^3)$ for all $s\in\R$, and the 
trace map is continuous from $H^2(\R^3)$ to $L^2(\Sigma)$ we conclude
\begin{align*}
 \overline{B_\lambda} h & = \lim_{n\to\infty} B_\mu h_n + \big((\lambda-\mu)(-\Delta-\lambda)^{-1}(-\Delta-\mu)^{-1}(h\delta_\Sigma)\big)|_\Sigma
\end{align*}
and hence the limit $\lim_{n\rightarrow\infty}B_\mu h_n$ exists and 
equals $\overline{B_\mu} h$. Using the continuity of $\gamma_\lambda-\gamma_\mu$ as a mapping from $L^2(\Sigma)$ into $H^2(\R^3)$, 
the continuity of the trace and \eqref{eq:integral2} we observe
\begin{equation}\label{eq:integral3}
 \begin{split}
  (\gamma_\lambda h-\gamma_\mu h)|_\Sigma
&=
\lim_{n\to\infty} (\gamma_\lambda h_n - \gamma_\mu h_n)|_\Sigma\\
&=
\lim_{n\to\infty} (B_\lambda h_n-B_\mu h_n)\\
&=
\overline{B_\lambda} h - \overline{B_\mu} h
 \end{split}
\end{equation}
for all $h\in\dom\overline{B_\lambda}=\dom\overline{B_\mu}$. From~\eqref{eq:integral3},~\eqref{eq:uc} and~\eqref{hdelta} we obtain
\begin{align*}
&\bigl\langle \big(\overline{B_\lambda} - \overline{B_\mu} \big)h,h\bigr\rangle_{L^2(\Sigma)}\\
&\qquad\qquad=
\bigl\langle (\gamma_\lambda h-\gamma_\mu h)|_\Sigma,h\bigr\rangle_{L^2(\Sigma)}  \\
&\qquad\qquad=
\sk{ \big[  (\lambda-\mu)(-\Delta-\lambda)^{-1}(-\Delta-\mu)^{-1}(h\delta_\Sigma)  \big]|_\Sigma,h}_{L^2(\Sigma)}  \\
&\qquad\qquad=
(\lambda-\mu) \sk{ (-\Delta-\lambda)^{-1}(-\Delta-\mu)^{-1}(h\delta_\Sigma) ,h\delta_\Sigma}_{2,-2}  \\
&\qquad\qquad=
(\lambda-\mu) \sk{ (-\Delta-\mu)^{-1}(h\delta_\Sigma) , (-\Delta-\lambda)^{-1}(h\delta_\Sigma)}_{L^2(\Sigma)}.
\end{align*}
Hence
\begin{align*}
\begin{split}
\lim_{\mu\to\lambda}\frac{ \sk{ \overline{B_\lambda} h,h}_{L^2(\Sigma)} - \sk{ \overline{B_\mu} h,h}_{L^2(\Sigma)} }{\lambda-\mu}
&=
\|(-\Delta-\lambda)^{-1}(h\delta_\Sigma)\|_{L^2(\Sigma)}^2\\
&=
\|\gamma_\lambda h\|_{L^2(\Sigma)}^2 .
\end{split}
\end{align*}
Since $\gamma_\lambda$ is an injective operator it follows that the function $\lambda\mapsto\sk{ \overline{B_\lambda} h,h}_{L^2(\Sigma)}$ is strictly increasing on $(-\infty,0)$, as its derivative is positive, i.e.,
\begin{align*}
 \big\langle \overline{B_\lambda} h, h \big\rangle_{L^2 (\Sigma)} < \big\langle \overline{B_\mu} h, h \big\rangle_{L^2 (\Sigma)}
\end{align*}
whenever $\lambda < \mu < 0$. From this and the min-max-principle for $\lambda < \mu < 0$ we obtain
\begin{align*}
 -\nu_k(\lambda) & = \min_{   \genfrac{}{}{0pt}{1}{  U\subseteq\dom\overline{B_\lambda}  }{\dim U=k}  }
\max_{   \genfrac{}{}{0pt}{1}{  h\in U  }{ \|h\|=1}  }   \sk{ - \overline{B_\lambda} h,h}_{L^2(\Sigma)}  \\
 & > \min_{   \genfrac{}{}{0pt}{1}{  U\subseteq\dom \overline{B_\mu} }{\dim U=k}  } \max_{   \genfrac{}{}{0pt}{1}{  h\in U  }{ \|h\|=1}  }   \sk{ - \overline{B_\mu} h,h}_{L^2(\Sigma)} = -\nu_k(\mu),
\end{align*}
where we have used that the operators $- \overline{B_\lambda}$ and $- \overline{B_\mu}$ are bounded from below; cf.~(ii). Thus $\nu_k (\lambda) < \nu_k (\mu)$ for $\lambda < \mu < 0$ and by continuity the same holds in the case $\lambda < \mu = 0$. This proves the remaining assertion in~(iv).
\end{proof}

\subsection{Well-definedness of the generalized trace}\label{subsec:Well-definedness}

In this subsection we verify that the definition of the generalized trace $u\vert_\Sigma$ in~\eqref{traceudef} is independent of the choice of $\lambda<0$. Observe first that if
\begin{equation}\label{decou1}
   u=u_c+\gamma_\lambda h,\qquad u_c\in
H^2(\R^3),\,\,\,h\in\dom\overline{B_\lambda},
\end{equation}
for some $\lambda<0$ then $h\in\dom \overline{B_\mu}$ for any $\mu<0$ by
Proposition~\ref{Proposition_Main}~(i) and
\begin{equation}\label{decou2}
u = v_c+\gamma_\mu h,
\qquad\text{where}\,\,\,v_c:= u_c+\gamma_\lambda h-\gamma_\mu h.
\end{equation}
It follows as in~\eqref{eq:uc} that $\gamma_\lambda h - \gamma_\mu h$ belongs to $H^2(\R^3)$, and hence also $v_c\in H^2(\R^3)$. Thus if $u$ admits the decomposition \eqref{decou1} with respect
to some $\lambda<0$ then $u$ admits the decomposition \eqref{decou2}
with respect to any $\mu<0$. Note also
that for fixed $\lambda<0$ both elements $u_c$ and $h$ in the decomposition~\eqref{decou1} are unique.

Let now $\lambda,\mu<0$ and assume that
\begin{equation}\label{uudeco}
u=u_c+\gamma_\lambda h=v_c+\gamma_\mu k
\end{equation}
with $u_c,v_c\in H^2(\R^3)$ and $h,k\in\dom 
\overline{B_\lambda}=\dom\overline{B_\mu}$.
Then it follows from the above considerations and the uniqueness
of the decompositions in \eqref{uudeco} that
\begin{equation}\label{voila}
  v_c=u_c+\gamma_\lambda h-\gamma_\mu h\quad\text{and}\quad h=k.
\end{equation}
Using~\eqref{eq:integral3} it follows from \eqref{voila} that
\begin{equation*}
\begin{split}
v_c\vert_\Sigma + \overline{B_\mu} k
& =
\bigl( u_c + \gamma_\lambda h-\gamma_\mu h\bigr)|_\Sigma + \overline{B_\mu} h\\
&=
u_c|_\Sigma + (\overline{B_\lambda} h-\overline{B_\mu} h) + \overline{B_\mu} h\\
&=
u_c|_\Sigma + \overline{B_\lambda} h .
\end{split}
\end{equation*}
This shows that the definition of the generalized trace in~\eqref{traceudef} is independent of the choice of~$\lambda$.

\section{Proofs of the main results}\label{sec4}

In this section we provide the complete proofs of the results in section~\ref{secMainRes}.

\subsection{Proof of Theorem~\ref{thm_KreinscheFormel}}

We start by proving assertion (i). Assume first that $\lambda\in\sigma_{\rm p}(\Op)$ for some 
$\lambda<0$, let 
$u\in \ker(\Op -\lambda)$, $u\not=0$,  and write
$u=u_c+\gamma_\lambda h$ with $u_c\in H^2(\R^3)$ and
$h\in\dom \overline{B_\lambda}$. Using the definition of
$\gamma_\lambda$ in \eqref{gammaop} we obtain
\begin{equation*}
\begin{split}
0 &= (\Op -\lambda) (u_c+\gamma_\lambda h) \\
  &= (-\Delta-\lambda) (u_c+\gamma_\lambda h) - \frac{1}{\alpha}  u|_\Sigma \cdot \delta_\Sigma \\
  &= (-\Delta-\lambda)u_c + \frac{1}{\alpha}(\alpha h -  u|_\Sigma) \delta_\Sigma.
\end{split}
\end{equation*}
Since $(- \Delta - \lambda) u_c \in L^2(\R^3)$ it follows $u_c = 0$. In particular, $0 \neq u = \gamma_\lambda h$, which implies $h \neq 0$. Moreover,
\begin{equation*}
  \alpha h = u|_\Sigma = (\gamma_\lambda h) |_\Sigma = \overline{B_\lambda} h,
\end{equation*}
that is, $h\in \ker(\alpha - \overline{B_\lambda})$. Since $u = \gamma_\lambda h$ it follows 
$$\ker(\Op-\lambda) 
\subseteq \gamma_\lambda \bigl( \ker(\alpha - \overline{B_\lambda}) \bigr).$$
Conversely, if 
$h\in\ker(\alpha - \overline{B_\lambda})$, $h\not=0$,
for some $\lambda<0$ set $u = \gamma_\lambda h$. Since $\gamma_\lambda$ is injective we obtain $u\not=0$ and
\begin{align*}
 u\vert_\Sigma = (\gamma_\lambda h) \vert_\Sigma = \overline{B_\lambda}h = \alpha h,
\end{align*}
and hence
\begin{equation*}
 (\cA_\alpha -\lambda) u= (-\Delta-\lambda) \gamma_\lambda h - \frac{1}{\alpha} u|_\Sigma \cdot \delta_\Sigma = h \delta_\Sigma - h \delta_\Sigma = 0.
\end{equation*}
From this we conclude $(\Op -\lambda) u=0$. Thus 
$$\gamma_\lambda \bigl( \ker(\alpha - \overline{B_\lambda}) \bigr) \subseteq  \ker(\Op-\lambda)$$ 
and 
$\lambda\in\sigma_{\rm p} (\Op)$. Since $\gamma_\lambda$ is continuous as a mapping from $L^2(\Sigma)$ into $L^2(\R^3)$
it follows that $\gamma_\lambda$ is an isomorphism between the spaces $\ker(\alpha - \overline{B_\lambda})$ and $\ker(\Op-\lambda)$.

Next we verify the resolvent formula \eqref{KreinscheFormel} in (ii) and, simultaneously, the selfadjointness of $\Op$. In the following for a given 
$\alpha \neq 0$ fix $\lambda_0 < 0$ such that $\alpha\not\in\sigma_{\rm p}( \overline{B_{\lambda_0}} )$; 
this is possible according to Proposition~\ref{Proposition_Main}~(iv). By item (i) we have 
\begin{equation*}
   \ker(\Op -\lambda_0)=\{0\}.
\end{equation*}
Let now $v \in L^2(\R^3)$ be arbitrary and define
\begin{equation}\label{udef}
  u = (\Opfree - \lambda_0)^{-1} v + \gamma_{\lambda_0} 
\bigl(\alpha - \overline{B_{\lambda_0}}\,\bigr)^{-1} 
\gamma_{\lambda_0}^*v \in L^2 (\R^3),
\end{equation}
and note that $(\alpha - \overline{B_{\lambda_0}})^{-1}$ 
is a bounded,
selfadjoint operator in $L^2(\Sigma)$; cf.\ Proposition~\ref{Proposition_Main}~(i) and~(ii).
Furthermore, as $(\Opfree - \lambda_0)^{-1}v \in H^2(\R^3)$ and
$(\alpha - \overline{B_{\lambda_0}})^{-1} \gamma_{\lambda_0}^*v 
\in \dom\overline{B_{\lambda_0}}$,
the trace $u \vert_\Sigma$ is well-defined in the sense 
of~\eqref{traceudef}. Making use of \eqref{goodtohave} we compute
\begin{equation}\label{soso}
  \begin{split}
  u |_\Sigma & = \big((\Opfree - \lambda_0)^{-1} v \big)|_\Sigma
  + \overline{B_{\lambda_0}} \bigl(\alpha - \overline{B_{\lambda_0}} \bigr)^{-1} 
\gamma_{\lambda_0}^*v  \\
  & = \left( I + \overline{B_{\lambda_0}} \bigl(\alpha - \overline{B_{\lambda_0}} \bigr)^{-1} \right)
\gamma_{\lambda_0}^*v \\
&= \alpha \left(\alpha  - \overline{B_{\lambda_0}} \right)^{-1} \gamma_{\lambda_0}^*v.
\end{split}
\end{equation}
From~\eqref{gammaop},~\eqref{soso} and the definition of $u$ 
in~\eqref{udef} we then conclude
\begin{equation*}
\begin{split}
  (\cA_\alpha-\lambda_0) u&=(-\Delta-\lambda_0)u-\frac{1}{\alpha} u|_\Sigma \cdot 
\delta_\Sigma\\
  & = v + \Big( \bigl( \alpha - \overline{B_{\lambda_0}}\,\bigr)^{-1} \gamma_{\lambda_0}^*v 
\Big)\cdot\delta_\Sigma - \frac{1}{\alpha}  u|_\Sigma\cdot\delta_\Sigma \\
  &= v
\end{split}
\end{equation*}
and hence $\cA_\alpha u=v+\lambda_0 u\in L^2(\R^3)$.
Thus we have $u\in\dom(\Op)$ and
\begin{align*}
(\Op-\lambda_0)^{-1} v = u =
(\Opfree - \lambda_0)^{-1}v + 
\gamma_{\lambda_0} \bigl(\alpha - \overline{B_{\lambda_0}}\,\bigr)^{-1}  \gamma_{\lambda_0}^*v.
\end{align*}
Since $v \in L^2 (\R^3)$ was arbitrary the 
identity~\eqref{KreinscheFormel} follows for $\lambda_0$.
In particular, since $(\alpha - \overline{B_{\lambda_0}})^{-1}$ is a 
bounded, selfadjoint operator in $L^2 (\Sigma)$, it follows that $(\Op - \lambda_0)^{-1}$
is bounded and selfadjoint in $L^2 (\R^3)$. This implies that $\lambda_0\in\rho(\Op)$ and that
$\Op$ is a selfadjoint operator in $L^2(\R^3)$.

Assume now that $\lambda\in\rho(\Op)\cap(-\infty,0)$ is arbitrary. Then $\alpha\in\rho(\overline{B_\lambda})$ by item (i) and Proposition~\ref{Proposition_Main}~(ii)
and the above arguments with $\lambda_0$ replaced by $\lambda$ yield the resolvent formula \eqref{KreinscheFormel} for all $\lambda\in\rho(\Op)\cap(-\infty,0)$.
The identity~\eqref{KreinscheFormel} also implies 
\begin{align*}
 \big\| (\Op-\lambda)^{-1} - (\Opfree - \lambda)^{-1} \big\| & = \big\| \gamma_{\lambda} \bigl(\alpha - \overline{B_{\lambda}}\,\bigr)^{-1}  \gamma_{\lambda}^* \big\| \\
 & \leq \| \gamma_{\lambda}\|^2  \big\| \bigl(\alpha - \overline{B_{\lambda}}\,\bigr)^{-1} \big\| \\
 & \leq \frac{\| \gamma_{\lambda}\|^2}{\alpha - \nu_1 (\lambda)}
\end{align*}
for all $\alpha > \nu_1(\lambda)$; cf.~Proposition~\ref{Proposition_Main}~(ii). It follows that 
the right-hand side converges to 0 as $\alpha\to+\infty$. This proves assertion~(ii).

In order to prove assertion~(iii) let first $\lambda = \lambda_0 \in \rho (\Op) \cap (- \infty, 0)$ be fixed. Then
\begin{align}\label{eq:ResFact}
(\Op - \lambda_0)^{-1} - (\Opfree - \lambda_0)^{-1} 
= \gamma_{\lambda_0} (\alpha - \overline{B_{\lambda_0}})^{-1} \gamma_{\lambda_0}^*.
\end{align}
Note that the identity~\eqref{goodtohave} implies that $\gamma_{\lambda_0}^*$ can also be regarded as a bounded operator from $L^2 (\R^3)$ to $H^1 (\Sigma)$ since the restriction map $H^2 (\R^3) \ni \p \mapsto \p |_\Sigma \in H^1 (\Sigma)$ is continuous (cf., e.g.,~\cite[Theorem 24.3]{BIN79}). In particular, it follows from the compactness of the embedding of $H^1 (\Sigma)$ into $L^2 (\Sigma)$ that
$\gamma_{\lambda_0}^*$ is compact. Since $(\alpha - \overline{B_{\lambda_0}})^{-1}$ is a bounded operator in $L^2 (\Sigma)$, the identity~\eqref{eq:ResFact} implies that the resolvent difference in~\eqref{eq:ResDiff} is compact for $\lambda = \lambda_0$. For an arbitrary $\lambda \in \rho (\Op) \cap \rho (\Opfree)$ a simple calculation yields
\begin{align*}
\begin{split}
(\Op - \lambda)^{-1}  &- (\Opfree - \lambda)^{-1}\\
&\qquad = U \big( (\Op - \lambda_0)^{-1} - (\Opfree - \lambda_0)^{-1} \big) V,
\end{split}
\end{align*}
where
$$
U=1 + (\lambda - \lambda_0) (\Opfree - \lambda)^{-1} \quad\text{and}\quad V=1 + (\lambda - \lambda_0) (\Op - \lambda)^{-1} 
$$
are bounded operators in $L^2 (\R^3)$. Now the claim follows from the assertion for $\lambda_0$. This proves~(iii).

\subsection{Proof of Theorem~\ref{thm:ResDiff}}

It suffices to prove the assertion of Theorem~\ref{thm:ResDiff} only for a fixed $$\lambda = \lambda_0 \in \rho (\Op) \cap (- \infty, 0).$$ 
Once it is established for $\lambda_0$ it follows for all $\lambda \in \rho (\Op) \cap \rho (\Opfree)$ with an argument 
as in the proof of Theorem~\ref{thm_KreinscheFormel}~(iii) and standard properties of singular values; cf. \cite[II.\S 2.2]{GK69}. When we denote by $- \Delta_{\rm 
LB}^\Sigma$ the Laplace--Beltrami operator in $L^2 (\Sigma)$ and write
$\Lambda := (I - \Delta_{\rm LB}^\Sigma)^{1/2}$ then $\Lambda$ is an 
isometric isomorphism between $H^1 (\Sigma)$ and $L^2 (\Sigma)$.
Moreover, $\Lambda^{-1}$ is a compact, selfadjoint operator in $L^2 
(\Sigma)$, whose singular values satisfy $s_k (\Lambda^{-1}) = O (1/k)$
as $k \to + \infty$; cf.~\cite[(5.39) and the text below]{A90}. 
Since $\gamma_{\lambda_0}^*$ is bounded from $L^2 (\R^3)$ to $H^1 (\Sigma)$ (see the proof of Theorem~\ref{KreinscheFormel}~(iii)) 
it follows that $\Lambda \gamma_{\lambda_0}^* : L^2 (\R^3) \to L^2 (\Sigma)$ is a 
bounded operator and from
\begin{align*}
\gamma_{\lambda_0}^* = \Lambda^{-1} \Lambda \gamma_{\lambda_0}^*
\end{align*}
we conclude $s_k (\gamma_{\lambda_0}^*) = O (1/k)$
as $k \to + \infty$; cf. \cite[II.\S 2.2]{GK69}.
As a consequence, also $\gamma_{\lambda_0} : L^2(\Sigma) \to L^2(\R^3)$ 
is a compact operator with $s_k (\gamma_{\lambda_0}) = O(1/k)$
as $k \to + \infty$. Moreover, with the help of~Corollary~2.2 in 
\cite[Chapter II]{GK69} we obtain
\begin{align}\label{eq:jminus2}
\begin{split}
 s_{3j-2} \li \gamma_{\lambda_0} (\alpha-\overline{B_\lambda} )^{-1} \gamma_{\lambda_0}^*\re & \leq s_{2j-1}\li \gamma_{\lambda_0} (\alpha-\overline{B_\lambda})^{-1} \re s_j( \gamma_{\lambda_0}^* ) \\
 & \leq s_j( \gamma_{\lambda_0} ) s_j\li (\alpha-\overline{B_\lambda} )^{-1} \re  s_j( \gamma_{\lambda_0}^* )
\end{split}
\end{align}
for all $j \in \N$. Due to these observations and
Proposition~\ref{Proposition_Main}~(iii) there exists $C = C (\lambda_0) > 0$ such that
\begin{align*}
s_j( \gamma_{\lambda_0} ) \leq \frac{C}{j} , \quad
s_j\li (\alpha - \overline{B_\lambda} )^{-1} \re  \leq 
\frac{C}{\ln j}, \quad \text{and} \quad
s_j( \gamma_{\lambda_0}^* ) \leq \frac{C}{j}
\end{align*}
hold for all $j\in\N$. From this the claim of the theorem follows for 
$\lambda = \lambda_0$. Indeed, for $j \geq 2$ with the help of~\eqref{eq:jminus2} we get
\begin{align*}
 s_{3j-2}\li \gamma_{\lambda_0} (\alpha-\overline{B_\lambda} )^{-1} \gamma_{\lambda_0}^* \re \leq \frac{C^3}{j^2\ln j} \leq \frac{27C^3}{(3j)^2\ln(3j)}
\end{align*}
since $\ln j=\frac{1}{3}\ln(j^3)\geq\frac{1}{3}\ln(3j)$. As
\begin{align*}
\begin{split}
s_{3j}\li \gamma_{\lambda_0} (\alpha - \overline{B_\lambda} )^{-1} \gamma_{\lambda_0}^*\re
&\leq s_{3j-1}\li \gamma_{\lambda_0} (\alpha - \overline{B_\lambda} )^{-1} \gamma_{\lambda_0}^*\re \\
&\leq s_{3j-2}\li \gamma_{\lambda_0} (\alpha - \overline{B_\lambda} )^{-1} \gamma_{\lambda_0}^* \re
\end{split}
\end{align*}
and
\begin{align*}
\frac{27C^3}{(3j)^2\ln(3j)}
\leq
\frac{27C^3}{(3j-1)^2\ln(3j-1)}
\leq
\frac{27C^3}{(3j-2)^2\ln(3j-2)}
\end{align*}
we observe 
\begin{align*}
s_k\li \gamma_{\lambda_0} (\alpha - \overline{B_\lambda} )^{-1} \gamma_{\lambda_0}^* \re
\leq \frac{27C^3}{k^2\ln k}
\end{align*}
for all $k\in\N$, $k\geq4$. This yields the assertion of the theorem.

\subsection{Proof of Theorem~\ref{thm:numberEV} and Corollary~\ref{cor:numberEVestimate}}

Let us first prove Theorem~\ref{thm:numberEV}.
For $\lambda \leq 0$ let us denote by $\nu_j (\lambda)$ the eigenvalues of the operator $\overline{B_\lambda}$, 
ordered nonincreasingly and counted with multiplicities; cf. Proposition~\ref{Proposition_Main} (iii). We remark that by Theorem~\ref{KreinscheFormel}~(i) 
and Proposition~\ref{Proposition_Main}~(iv)
the number $N_\alpha$ of negative eigenvalues of $\Op$ counted with multiplicities coincides with the number of 
eigenvalues of $\overline B_0$ larger than $\alpha$, counted with multiplicities.
Moreover, 
let $\T$ be a circle of radius $R = \frac{L}{2 \pi}$, where $L$ is the length of $\Sigma$. We denote by $B_\lambda^\T$ the analog of $B_\lambda$ where $\Sigma$ is replaced by the circle $\T$, and by $\nu_j^\T (\lambda)$ the eigenvalues 
of its closure. From \eqref{eq:perturbation} with $\lambda = 0$ it follows with the min-max principle that
\begin{align*}
 \nu_j^\T (0) - \Vert D_0\Vert \leq \nu_j (0) \leq \nu_j^\T (0) + \Vert D_0\Vert , \quad j = 1, 2, \dots.
\end{align*}
Taking into account \eqref{eq:D0est} we obtain
\begin{align}\label{eq:EVrelation}
 \nu_j^\T (0) - d_\Sigma \leq \nu_j (0) \leq \nu_j^\T (0) + d_\Sigma, \quad j = 1, 2, \dots.
\end{align}

Assume first that $\alpha - d_\Sigma \geq \frac{\ln (4 R)}{2 \pi}$. For $\lambda<0$ and $j = 1, 2, \dots$ we obtain from Proposition~\ref{Proposition_Main}~(iv),~\eqref{eq:EVrelation}, and Lemma~\ref{Lemma_T_0}~(ii)
\begin{align*}
\nu_j (\lambda) < \nu_j (0) \leq \nu_1 (0) \leq \nu_1^\T (0) + d_\Sigma = \frac{\ln (4 R)}{2 \pi} + d_\Sigma \leq \alpha.
\end{align*}
In particular, $\alpha \notin \sigma_{\rm p} (\overline{B_\lambda})$
for all $\lambda < 0$. From this and Theorem~\ref{thm_KreinscheFormel}~(i) it follows
$\lambda \notin \sigma_{\rm p} (\Op)$ for all $\lambda < 0$, hence $N_\alpha = 0$.

Assume now $\alpha + d_\Sigma \in I_r$ for some $r \geq 0$ and $\alpha - d_\Sigma \in I_l$ for some $l \geq 0$. 
By means of Lemma~\ref{Lemma_T_0}~(ii) this implies
\begin{align}\label{eq:ineqAlpha}
 \nu_{2 r + 2}^\T (0) \leq \alpha + d_\Sigma < \nu_{2r+1}^\T (0)
\end{align}
and
\begin{align}\label{eq:ineqAlpha2}
 \nu_{2 l + 2}^\T (0) \leq \alpha - d_\Sigma < \nu_{2 l + 1}^\T (0).
\end{align}
From~\eqref{eq:ineqAlpha},~\eqref{eq:ineqAlpha2} and~\eqref{eq:EVrelation} it follows
\begin{align}\label{eq:ineqResult}
 \nu_{2 l + 2} (0) \leq \nu_{2 l + 2}^\T (0) + d_\Sigma \leq \alpha < \nu_{2 r + 1}^\T (0) - d_\Sigma \leq \nu_{2 r + 1} (0).
\end{align}
Due to Proposition~\ref{Proposition_Main}~(iv) the functions $\lambda \mapsto \nu_j (\lambda)$ are continuous 
and strictly increasing and satisfy $\nu_j (\lambda) \to - \infty$ as $\lambda \to - \infty$, $j = 1, 2, \dots$. 
Thus by~\eqref{eq:ineqResult} for each $j \leq 2 r + 1$ there exists precisely one $\lambda_j < 0$ such that
$\nu_j (\lambda_j) = \alpha$. 
From Theorem~\ref{thm_KreinscheFormel}~(i) we conclude that each such $\lambda_j$ is an eigenvalue of $\Op$ and hence 
we obtain the estimate
\begin{align*}
 2 r + 1 \leq N_\alpha.
\end{align*}
In the same way \eqref{eq:ineqResult} implies that 
for any $j \geq 2 l + 2$ there exists no $\lambda < 0$ such that $\nu_j (\lambda) = \alpha$ and
that for each $j \in \{ k : 2 r + 2 \leq k \leq 2 l + 1\}$ there exists at most one $\lambda_j < 0$ such that $\nu_j (\lambda_j) = \alpha$.
Theorem~\ref{thm_KreinscheFormel}~(i) yields that 
each such $\lambda_j$ is an eigenvalue of $\Op$ and therefore
\begin{align*}
 N_\alpha\leq 2l+1.
\end{align*}

In the remaining case $\alpha + d_\Sigma \in I_r$ with $r = - 1$ it is clear that
\begin{align*}
 2 r + 1 = - 1 \leq N_\alpha,
\end{align*}
and the upper estimate for $N_\alpha$ follows as above. This completes the proof of the theorem.

Let us now turn to the proof of the corollary. As in Theorem~\ref{thm:numberEV} let $r$ and $l$ such that $\alpha+d_\Sigma\in I_r$ and $\alpha-d_\Sigma\in I_l$. The condition $\alpha + d_\Sigma < \frac{\ln(4R)}{2\pi}-\frac{1}{\pi}$ ensures $1\leq r \leq l$. The proof is based on the estimates
\begin{align}\label{eq:Knuth}
 \ln k + \gamma + \frac{1}{2k} - \frac{1}{12k^2} < H_k < \ln k + \gamma + \frac{1}{2k} - \frac{1}{12k^2} + \frac{1}{120k^4}
\end{align}
for the harmonic sum $H_k = \sum_{j=1}^k \frac{1}{j}$, $k\geq1$, see e.g.~\cite[(9.89)]{GKP89}. Since $\sum_{j=1}^k \frac{1}{2j-1} = H_{2k} - \frac{1}{2} H_k$ it follows from~\eqref{eq:Knuth}
\begin{align*}
 \sum_{j=1}^k \frac{1}{2j-1} & > \ln(2k) + \gamma + \frac{1}{4k} - \frac{1}{48k^2} - \frac{1}{2} \Big(\ln k + \gamma + \frac{1}{2k} - \frac{1}{12k^2} + \frac{1}{120k^4}  \Big)  \\
 & = \frac{\ln k + \ln4 + \gamma}{2}  + \frac{1}{48k^2} - \frac{1}{240k^4} \\
 & >  \frac{\ln k + \ln4 + \gamma}{2}.
\end{align*}
Hence $\alpha-d_\Sigma\in I_l$ implies
\begin{align*}
\alpha-d_\Sigma
&< \frac{\ln(4R)}{2\pi} - \frac{1}{\pi} \sum_{j=1}^l \frac{1}{2j-1}  
< \frac{\ln(4R)}{2\pi} - \frac{  \ln l + \ln4 + \gamma  }{2\pi}
\end{align*}
and therefore
\begin{align}\label{est:for_upper}
\ln l  <  -2\pi(\alpha-d_\Sigma) + \ln R - \gamma  .
\end{align}
Using $N_\alpha\leq 2l+1$ from Theorem~\ref{thm:numberEV} and the estimate \eqref{est:for_upper} we get
\begin{align*}
N_\alpha - 2R e^{-2\pi(\alpha-d_\Sigma)-\gamma}
\leq
2l+1 - 2 e^{-2\pi(\alpha-d_\Sigma) + \ln R-\gamma}  
<
2l+1 - 2 e^{\ln l} = 1
\end{align*}
which yields the upper estimate for $N_\alpha$ in \eqref{eq:numberEVestimate}.

For the lower estimate in \eqref{eq:numberEVestimate} we deduce
from \eqref{eq:Knuth} the estimate
\begin{align*}
\begin{split}
&\sum_{j=1}^k \frac{1}{2j-1}\\
&\quad <
\ln(2k) + \gamma + \frac{1}{4k} - \frac{1}{48k^2} + \frac{1}{1920k^4} -
\frac{1}{2}\left(  \ln k + \gamma + \frac{1}{2k} - \frac{1}{12k^2}  \right)  \\
&\quad =
\frac{\ln k + \ln4 + \gamma}{2} +
\frac{1}{48k^2} + \frac{1}{1920k^4}\\
&\quad <
\frac{\ln k + \ln4 + \gamma + \frac{1}{23k^2}}{2} .
\end{split}
\end{align*}
Hence $\alpha+d_\Sigma\in I_r$ implies
\begin{align*}
\begin{split}
\alpha+d_\Sigma &
\geq \frac{\ln(4R)}{2\pi} - \frac{1}{\pi} \sum_{j=1}^{r+1} \frac{1}{2j-1}\\  
&> \frac{\ln(4R)}{2\pi} - \frac{  \ln(r+1) + \ln4 + \gamma + \frac{1}{ 23(r+1)^2 }  }{2\pi}
\end{split}
\end{align*}
and therefore
\begin{align}\label{est:for_lower}
\ln(r+1) + \frac{1}{ 23(r+1)^2 }  > -2\pi(\alpha+d_\Sigma) + \ln R - \gamma.
\end{align}
Using $N_\alpha\geq 2r+1$ from Theorem~\ref{thm:numberEV} and the estimate \eqref{est:for_lower} we get
\begin{align*}
N_\alpha - 2R e^{-2\pi(\alpha+d_\Sigma)-\gamma}
&\geq
2r+1 - 2 e^{-2\pi(\alpha+d_\Sigma)+\ln R-\gamma}  \\
&>
2r+1 - 2 e^{\ln(r+1) + \frac{1}{ 23(r+1)^2 } }  \\
&=
2(r+1)-1 - 2(r+1)e^{ \frac{1}{ 23(r+1)^2 } }  \\
&=
2(r+1) \Big(  1 - e^{ \frac{1}{ 23(r+1)^2 } } \Big) -1  =: g(r).
\end{align*}
As $g'(r)>0$ for all $r\geq 1$, the minimum of $g$ for $r \geq 1$ is attained at $r = 1$. Hence
\begin{align*}
N_\alpha - 2R e^{-2\pi(\alpha+d_\Sigma)-\gamma}
>
4 \Big(  1 - e^{ \frac{1}{92} }  \Big) -1 ,
\end{align*}
which gives the lower estimate in \eqref{eq:numberEVestimate}.

\subsection{Proof of Theorem~\ref{thm:EVmaximizer}}

The proof of Theorem~\ref{thm:EVmaximizer} follows the ideas of~\cite{E05,EHL06}. Suppose that $\Sigma$ is not a circle. 
Then the strict
inequality
\begin{align}\label{ineqExnerEtAl}
  \int_0^L | \sigma(s+u)-\sigma(s)| \;ds <  \frac{L^2}{\pi} \sin 
\frac{\pi u}{L}, \quad u \in (0,L),
\end{align}
holds,
where $\sigma$ is identified with its $L$-periodic extension to all of 
$\R$. For $u \in (0, \frac{L}{2} ]$ the inequality \eqref{ineqExnerEtAl} 
follows from~\cite[Theorem~2.2 and Proposition~2.1]{EHL06}. As every $u 
\in ( \tfrac{L}{2} , L )$ can be written as $u=L-v$ with $v\in(0, 
\tfrac{L}{2} )$, the substitution $t=s-v$ and the periodicity of 
$\sigma$ yield for $u \in ( \tfrac{L}{2} , L )$
\begin{equation*}
 \begin{split}
 \int_0^L | \sigma(s+u)-\sigma(s)| \;ds
& =
 \int_0^L | \sigma(s-v) - \sigma(s) | \;ds\\
&=
\int_0^L | \sigma(t) - \sigma(t+v)  | \;dt\\
&<
\frac{L^2}{\pi} \sin \frac{\pi v}{L}\\
&=
\frac{L^2}{\pi} \sin \frac{\pi u}{L} , 
 \end{split}
\end{equation*}
i.e.,the estimate \eqref{ineqExnerEtAl} holds for all $u\in(0,L)$.

In the following denote by  $\lambda_1 = \min \sigma ( -\Delta_{\T,\alpha} ) < 0$
the smallest eigenvalue of $- \Delta_{\T,\alpha}$
(cf.~Corollary~\ref{cor:circle}) and let $\nu_1^\T(\lambda_1)$ be
the largest eigenvalue of $\overline{ B_{\lambda_1}^\T }$.
By 
Theorem~\ref{thm_KreinscheFormel}~(i) we have $\alpha \in 
\sigma_{\rm p} (\overline{B_{\lambda_1}^\T})$ and, in particular, 
$\alpha \leq \nu_1^\T (\lambda_1)$.

We claim that 
\begin{equation}\label{nuclaim}
\nu_1^\T (\lambda_1) < \nu_1 (\lambda_1)
\end{equation}
holds. In order to see this note first that
\eqref{eq:perturbation} implies
\begin{align}\label{eq:bla}
  \overline{B_{\lambda_1}} = D_{\lambda_1} + J^* 
\overline{B_{\lambda_1}^\T} J ,
\end{align}
where $J : L^2 (\Sigma) \to L^2 (\T)$ is the unitary mapping given in 
\eqref{eq:J} and the compact operator $D_{\lambda_1}$ in $L^2 (\Sigma)$ 
is given by
\begin{align*}
  \big( D_{\lambda_1} h \big)( \sigma(t)) = \int_0^L h(\sigma(s)) \left[ 
\frac{e^{-\sqrt{-\lambda_1}|\sigma(t)-\sigma(s)|}}{4\pi|\sigma(t)-\sigma(s)|} 
- 
\frac{e^{-\sqrt{-\lambda_1}|\tau(t)-\tau(s)|}}{4\pi|\tau(t)-\tau(s)|}\right] 
d s
\end{align*}
for $h \in L^2 (\Sigma)$. It follows from 
Lemma~\ref{Lemma_T_lambda}~(iv) and~\eqref{eq:bla} that for the constant 
function $h = \frac{1}{\sqrt{L}}$ on $\Sigma$ (which implies $\|h\|_{L^2(\Sigma)}=1$)
we have
\begin{equation}\label{eq:bisschenwas}
 \begin{split}
 \sk{ \overline{B_{\lambda_1}} h,h}_{L^2(\Sigma)}
&=
\sk{D_{\lambda_1} h,h}_{L^2(\Sigma)} + \sk{ \overline{B_{\lambda_1}^\T} 
Jh,Jh}_{L^2(\mathcal{T})}\\
&=
\sk{D_{\lambda_1} h,h}_{L^2(\Sigma)} + \nu_1^\T (\lambda_1). 
 \end{split}
\end{equation}

Our aim is to estimate the term $\sk{D_{\lambda_1} h,h}_{L^2(\Sigma)}$. 
For this purpose we define the function
\begin{align*}
  G (x) = \frac{e^{-\sqrt{-\lambda_1}x}}{4 \pi x}, \qquad x>0.
\end{align*}
It is easy to see that $G$ is strictly monotone decreasing and 
convex. Hence~\eqref{ineqExnerEtAl} and the monotonicity of $G$ imply
\begin{align}\label{eq:decreasing}
  G \li \frac{L}{\pi} \sin \frac{\pi u}{L} \re < G \li \frac{1}{L} 
\int_0^L | \sigma(s+u)-\sigma(s)| ds \re
\end{align}
for each $u \in (0,L)$. Using Jensen's Inequality, see e.g.\ 
\cite[Theorem 3.3]{R70}, the convexity of $G$ implies
\begin{align}\label{eq:Jensen}
  G \li  \frac{1}{L} \int_0^L | \sigma(s+u)-\sigma(s)| ds \re \leq 
\frac{1}{L} \int_0^L G (| \sigma(s+u)-\sigma(s)|) ds .
\end{align}
Combining \eqref{eq:decreasing} and \eqref{eq:Jensen} we observe
\begin{align}\label{Ungleichung1}
\begin{split}
  0 &< \int_0^L  \li \int_0^L G (| \sigma(s+u)-\sigma(s)|) \;ds - L G 
\li \frac{L}{\pi} \sin \frac{\pi u}{L} \re  \re du  \\
  & = \int_0^L  \int_0^L G (| \sigma(s+u)-\sigma(s)|)  - G \li 
\frac{L}{\pi} \sin \frac{\pi u}{L} \re   du ds.
\end{split}
\end{align}
Moreover for each $s \in (0, L)$ with the substitution $t = s+u$ we get
\begin{align*}
  \int_0^L G &(| \sigma(s+u)-\sigma(s)|)  - G \li \frac{L}{\pi} \sin 
\frac{\pi u}{L} \re   du  \\
  & = \int_s^{L+s} G (| \sigma(t)-\sigma(s)|)  - G \li \frac{L}{\pi} 
\sin \frac{\pi(t-s)}{L} \re   dt \\
  & = \int_s^L G (| \sigma(t)-\sigma(s)|)  - G \li \frac{L}{\pi} \sin 
\frac{\pi(t-s)}{L} \re   dt \\
  &\quad + \int_0^s G (| \sigma(t+L)-\sigma(s)|)  - G \li \frac{L}{\pi} 
\sin \frac{\pi(t+L-s)}{L} \re   dt \\
  & = \int_s^L G (| \sigma(t)-\sigma(s)|)  - G \li \frac{L}{\pi} \sin 
\frac{\pi(t-s)}{L} \re   dt \\
  &\quad + \int_0^s G (| \sigma(t)-\sigma(s)|)  - G \li \frac{L}{\pi} 
\sin \frac{\pi(s-t)}{L} \re   dt \\
  & = \int_0^L G (| \sigma(t)-\sigma(s)|)  - G \li \frac{L}{\pi} \sin 
\frac{\pi|t-s|}{L} \re   dt.
\end{align*}
Therefore \eqref{Ungleichung1} can be rewritten as
\begin{align*}
  0 & < \int_0^L  \int_0^L G (| \sigma(t)-\sigma(s)|)  - G \li 
\frac{L}{\pi} \sin \frac{\pi|t-s|}{L} \re   dt ds.
\end{align*}
 From the last equality and~\eqref{eq:chord} (with $\sigma$ replaced by $\tau$) we conclude
\begin{align*}
\begin{split}
\sk{D_{\lambda_1} h,h}_{L^2(\Sigma)}
&= \frac{1}{L} \int_{0}^{L}     \int_{0}^{L}
\frac{e^{-\sqrt{-\lambda_1}|\sigma(t)-\sigma(s)|}}{4\pi|\sigma(t)-\sigma(s)|} 
- \frac{e^{-\sqrt{-\lambda_1}|\tau(t)-\tau(s)|}}{4\pi|\tau(t)-\tau(s)|} 
\,ds   \,dt \\
&> 0
\end{split}
\end{align*}
for the constant function $h = \frac{1}{\sqrt{L}}$. Hence~\eqref{eq:bisschenwas} leads to
\begin{align*}
  \sk{ \overline{B_{\lambda_1}} h,h}_{L^2(\Sigma)} > \nu_1^\T (\lambda_1)
\end{align*}
for the constant function $h = \frac{1}{\sqrt{L}}$ and hence \eqref{nuclaim} follows. In particular,
\begin{align*}
\alpha \leq \nu_1^\T (\lambda_1) < \nu_1 (\lambda_1).
\end{align*}
As the function $\lambda \mapsto \nu_1 (\lambda)$ is continuous and 
strictly increasing on $(-\infty, 0]$ by 
Proposition~\ref{Proposition_Main}~(iv) and $\nu_1(\lambda) \to -\infty$ 
as $\lambda \to -\infty$, there exists $\lambda_2 < \lambda_1$ such 
that $\alpha = \nu_1(\lambda_2)$. By 
Theorem~\ref{thm_KreinscheFormel}~(i) $\lambda_2$ is an eigenvalue of 
$\Op$. Thus
\begin{align*}
  \min \sigma (\Op) \leq \lambda_2 < \lambda_1 = 
\min \sigma (- \Delta_{\T, \alpha}),
\end{align*}
which completes the proof of Theorem~\ref{thm:EVmaximizer}.

\subsection{Proof of Theorem~\ref{nrthm}}\label{sec5}

Consider the scattering pair $\{\Opfree,\Op\}$ with $\alpha\in\R\setminus\{0\}$ and fix some 
$\eta<0$ such that $0\in\rho(\overline{B_\eta}-\alpha)$, which is possible
according to Proposition~\ref{Proposition_Main}~(ii) and (iv). As in \eqref{sssi} and \eqref{ttt} consider the symmetric operator
\begin{equation*}
Su= -\Delta u,\qquad \dom S=\bigl\{u\in H^2(\R^3):u\vert_\Sigma=0\bigr\},
\end{equation*}
and the operator
\begin{equation*}
 Tu=-\Delta u-h\delta_\Sigma,\qquad
 \dom T=H^2(\R^3)\,\dot +\,\bigl\{\gamma_\eta h:h\in\dom \overline{B_\eta}\bigr\}, 
\end{equation*}
where $\gamma_\eta h=(-\Delta-\eta)^{-1}(h\delta_\Sigma)$ is as in \eqref{gammaop}. Then $\overline T=S^*$ according to Proposition~\ref{qbtprop}.
Now we slightly modify the boundary maps in Proposition~\ref{qbtprop} such that Theorem~\ref{scatti} can be applied directly to the pair $\{\Opfree,\Op\}$.
More precisely, we claim that $\{L^2(\Sigma),\Gamma_0,\Gamma_1\}$, where
\begin{equation}\label{qbtttt2}
 \Gamma_0 u= h\quad\text{and}\quad\Gamma_1 u= u_c\vert_\Sigma+(\overline{B_\eta}-\alpha)h,\quad u=u_c+\gamma_\eta h\in\dom T,
\end{equation}
is a quasi boundary triple for $S^*$ such that
\begin{equation}\label{passtso}
\Opfree=T\upharpoonright\ker\Gamma_0\quad\text{and}\quad \Op=T\upharpoonright\ker\Gamma_1.
\end{equation}
The $\gamma$-field and Weyl function corresponding to $\{L^2(\Sigma),\Gamma_0,\Gamma_1\}$ are given by
\begin{equation}\label{gm2}
 \gamma(\lambda) h =(-\Delta-\lambda)^{-1}(h\delta_\Sigma)  \quad\text{and} \quad M(\lambda) h= N(\lambda)h+(\overline{B_\eta}-\alpha)h,
\end{equation}
where $\lambda\in\C\setminus[0,\infty)$, $h\in\dom \overline{B_\eta}$, and the function $N$ is as in \eqref{nnn}.

In fact the identities in \eqref{passtso} hold by construction and Proposition~\ref{qbtprop}. In order to verify the abstract Green identity for the boundary maps in \eqref{qbtttt2}
recall from \eqref{ggg} in the proof of Proposition~\ref{qbtprop} that for $u,v\in\dom T$ such that $u=u_c+\gamma_\eta h$ and $v=v_c+\gamma_\eta k$ the identity
\begin{equation*}
 \langle  Tu,v\rangle_{L^2 (\R^3)}  - \langle u,Tv\rangle_{L^2 (\R^3)} =\langle u_c\vert_\Sigma,k\rangle_{L^2(\Sigma)}-\langle h,v_c\vert_\Sigma\rangle_{L^2(\Sigma)}
\end{equation*}
holds. Since $(\overline{B_\eta}-\alpha)$ is a selfadjoint operator in $L^2(\Sigma)$ we have
\begin{equation*}
\begin{split}
 &\langle u_c\vert_\Sigma,k\rangle_{L^2(\Sigma)}-\langle h,v_c\vert_\Sigma\rangle_{L^2(\Sigma)}\\
 &\qquad \qquad=
 \bigl\langle u_c\vert_\Sigma+(\overline{B_\eta}-\alpha)h,k\bigr\rangle_{L^2(\Sigma)}-\bigl\langle h,v_c\vert_\Sigma+(\overline{B_\eta}-\alpha)k\bigr\rangle_{L^2(\Sigma)}\\
 &\qquad\qquad=\langle \Gamma_1 u,\Gamma_0 v\rangle_{L^2(\Sigma)}-\langle \Gamma_0 u,\Gamma_1 v\rangle_{L^2(\Sigma)}
 \end{split}
\end{equation*}
and hence the Green identity is valid. The same argument as in the proof of Proposition~\ref{qbtprop} shows 
that the range of the mapping $u\mapsto (\Gamma_0 u,\Gamma_1 u)^\top$ is dense in $L^2(\Sigma)\times L^2(\Sigma)$. Hence $\{L^2(\Sigma),\Gamma_0,\Gamma_1\}$
is a quasi boundary triple for $S^*$. Since $\Gamma_0$ is the same map as in Proposition~\ref{qbtprop} the corresponding $\gamma$-field has the same form
as in Proposition~\ref{qbtprop}. The form of the Weyl function in \eqref{gm2} follows from
\begin{equation*}
 M(\eta) h=\Gamma_1 \gamma(\eta) h=\Gamma_1(-\Delta-\eta)^{-1}(h\delta_\Sigma)=(\overline{B_\eta}-\alpha)h
\end{equation*}
for $h\in\ran\Gamma_0=\dom\overline{B_\eta}$ and \eqref{nnn} in the same way as in the proof of Proposition~\ref{qbtprop}; cf.~\eqref{eq:uc}, \eqref{eq:integral2},
and Remark~\ref{remchen}.

Now we complete the proof of Theorem~\ref{nrthm}. Consider the quasi boundary triple $\{L^2(\Sigma),\Gamma_0,\Gamma_1\}$ in \eqref{qbtttt2}.
It follows from \eqref{gm2}, \eqref{gammaop} and the proof of Theorem~\ref{thm:ResDiff} that
\begin{equation}\label{useit33}
 \overline{\gamma(\eta)}=\gamma_\eta\in\sS_2\bigl(L^2(\Sigma),L^2(\R^3)\bigr).
\end{equation}
Moreover, since $\eta<0$ was chosen such that $0\in\rho(\overline{B_\eta}-\alpha)$ it is clear that the operator
$M(\eta)^{-1}=(\overline{B_\eta}-\alpha)^{-1}$ is  bounded in $L^2(\Sigma)$. Note also that
\begin{equation*}
 \overline{\text{\rm Im}\, M(\lambda)}=\text{\rm Im}\, N(\lambda),\qquad\lambda\in\C\setminus [0,\infty),
\end{equation*}
holds by~\eqref{gm2}. Hence the assumptions in Theorem~\ref{scatti} are satisfied and the assertions (i), (iii), and (iv) in Theorem~\ref{nrthm} follow.
Observe that by \eqref{gm2} and \eqref{m1}
\begin{align*}
N(\lambda)
&= (\lambda-\eta)\gamma(\eta)^*(\Opfree-\eta)(\Opfree-\lambda)^{-1}\overline{\gamma(\eta)} \\
&= (\lambda-\eta)\gamma(\eta)^*\overline{\gamma(\eta)}
+ (\lambda-\eta)^2\gamma(\eta)^*(\Opfree-\lambda)^{-1}\overline{\gamma(\eta)}
\end{align*}
holds for $\lambda\in\C\setminus [0,\infty)$. Therefore \eqref{useit33} and \cite[Proposition 3.14]{BW83} yield that the limit $N(\lambda+i0)$ exists in the Hilbert-Schmidt
norm for a.e. $\lambda\in[0,\infty)$, that is, assertion (ii) in Theorem~\ref{nrthm} holds. This completes the proof.

\appendix
\section{Quasi boundary triples and their Weyl functions}\label{appi}

In this appendix we briefly review the abstract notions of quasi boundary triples and
their Weyl functions from extension theory of symmetric operators in Hilbert spaces, and relate them to the Schr\"{o}dinger operators $\Opfree$ and $\Op$.
Furthermore, we recall a representation formula for the scattering matrix in terms of the Weyl function of a quasi boundary triple from \cite{BMN15},
which is the main ingredient in the proof of Theorem~\ref{nrthm}.
For more details on quasi boundary triples and their Weyl functions we refer the reader to \cite{BL07,BL12}, and 
for generalized and ordinary boundary triples to \cite{BGP08,DM91,DM95}.
 
\begin{defn}\label{qbtdeffi}
Let $S$ be a densely defined, closed, symmetric operator in a Hilbert space $(\sH, \langle \cdot, \cdot \rangle_\sH)$ 
and assume that $T$ is a linear operator in $\sH$  such that $\overline T=S^*$.
A triple $\{\cG,\Gamma_0,\Gamma_1\}$ is a {\em quasi boundary triple} for $S^*$ if $(\cG, \langle \cdot, \cdot \rangle_\cG)$ is a Hilbert space and 
$\Gamma_0,\Gamma_1:\dom T\rightarrow\cG$ are 
linear mappings such that the following holds.
\begin{enumerate}
\item For all $u,v\in\dom T$ one has
 \begin{equation*}
  \langle Tu,v\rangle_\sH-\langle u,Tv\rangle_\sH=\langle\Gamma_1 u,\Gamma_0 v\rangle_\cG-\langle\Gamma_0 u,\Gamma_1 v\rangle_\cG.
 \end{equation*}
\item The range of the mapping $(\Gamma_0,\Gamma_1)^\top:\dom T\rightarrow\cG\times\cG$ is dense.
\item The operator $A_0:=T\upharpoonright\ker\Gamma_0$ is selfadjoint in $\sH$.
\end{enumerate}
\end{defn}

If $\{\cG,\Gamma_0,\Gamma_1\}$ is a quasi boundary triple for $\overline T=S^*$ then
\begin{equation*}
 S=T\upharpoonright\bigl(\ker\Gamma_0\cap\ker\Gamma_1\bigr).
\end{equation*}
Moreover, if $\ran\Gamma_0=\cG$ then $\{\cG,\Gamma_0,\Gamma_1\}$ is a generalized boundary triple in the sense of \cite[Section 6]{DM95}, and if 
$\ran(\Gamma_0,\Gamma_1)^\top=\cG\times\cG$ then $\{\cG,\Gamma_0,\Gamma_1\}$ is an ordinary boundary triple; cf. \cite{BGP08,DM91}. In the latter case
it follows that $T=S^*$ and hence the abstract Green identity in Definition~\ref{qbtdeffi}~(i) holds for all $u,v\in\dom S^*$. We remark that for an ordinary boundary triple
condition (iii) in Definition~\ref{qbtdeffi} is automatically satisfied. 

A quasi boundary triple $\{\cG,\Gamma_0,\Gamma_1\}$ for $\overline T=S^*$ is a useful tool to describe the extensions of $S$ which are contained in $T$ via
abstract boundary conditions in the auxiliary Hilbert space $\cG$. However,
in this context it is important to note that not all selfadjoint extensions of $S$ in $\sH$ are covered, but only those which are also restrictions of $T$. Furthermore,
a selfadjoint parameter $\Theta$ in $\cG$ does not automatically lead to a selfadjoint extension via
\begin{equation}\label{atti}
 A_\Theta:=T\upharpoonright\ker(\Gamma_1-\Theta\Gamma_0),
\end{equation}
as one is used to from the theory of ordinary boundary triples. In general $A_\Theta$ in \eqref{atti} is only symmetric in $\sH$, not necessarily closed, 
and one has to impose additional conditions on $\Theta$ 
or on other involved objects to ensure selfadjointness of the extension $A_\Theta$, see, e.g. \cite{BL07,BL12}.

Next we recall \cite[Theorem 6.11]{BL12} which is very useful for the construction of quasi boundary triples and provides a method to determine the 
adjoint of a symmetric operator. 

\begin{thm}\label{thm:A2}
Let $T$ be a linear operator in a Hilbert space $(\sH, \langle \cdot, \cdot \rangle_\sH)$, let 
$(\cG, \langle \cdot, \cdot \rangle_\cG)$ be a Hilbert space, and assume that $\Gamma_0,\Gamma_1:\dom T\rightarrow\cG$ are linear mappings
such that the following holds.
\begin{enumerate}
\item For all $u,v\in\dom T$ one has
\begin{equation*}
 \langle Tu,v\rangle_\sH-\langle u,Tv\rangle_\sH=\langle\Gamma_1 u,\Gamma_0 v\rangle_\cG-\langle\Gamma_0 u,\Gamma_1 v\rangle_\cG.
\end{equation*}
\item $\ran(\Gamma_0,\Gamma_1)^\top$ is dense in $\cG\times\cG$ and $\ker\Gamma_0\cap\ker\Gamma_1$ is dense in $\sH$.
\item There exists a selfadjoint operator $A_0$ in $\sH$ such that $A_0\subset T\upharpoonright\ker\Gamma_0$.
\end{enumerate}
Then $S:=T\upharpoonright(\ker\Gamma_0\cap\ker\Gamma_1)$ is a densely defined, closed, symmetric operator in $\sH$ such that $\overline T=S^*$, and 
$\{\cG,\Gamma_0,\Gamma_1\}$ is a quasi boundary triple for $S^*$ with $A_0= T\upharpoonright\ker\Gamma_0$.
\end{thm}

Next we recall the notion of the $\gamma$-field and Weyl function associated to a quasi boundary triple $\{\cG,\Gamma_0,\Gamma_1\}$ for $\overline T=S^*$.
First of all it follows from the direct sum decomposition $\dom T=\dom A_0\dot +\ker(T-\lambda)$, $\lambda\in\rho(A_0)$, and $\dom A_0=\ker\Gamma_0$ 
that the restriction of the boundary map $\Gamma_0$ onto $\ker(T-\lambda)$ is invertible. The inverse 
\begin{align*}
 \gamma(\lambda)=\bigl(\Gamma_0\upharpoonright\ker(T-\lambda)\bigr)^{-1},\qquad\lambda\in\rho(A_0),
\end{align*}
is a densely defined operator from $\cG$ into $\sH$. The function $\lambda\mapsto\gamma(\lambda)$ is called the $\gamma$-field associated to 
$\{\cG,\Gamma_0,\Gamma_1\}$. The Weyl function $M$ associated to 
$\{\cG,\Gamma_0,\Gamma_1\}$ is defined by
\begin{align*}
 M(\lambda)=\Gamma_1\bigl(\Gamma_0\upharpoonright\ker(T-\lambda)\bigr)^{-1},\qquad\lambda\in\rho(A_0). 
\end{align*}
The values $M(\lambda)$ of the Weyl function are densely defined operators in $\cG$, which may be unbounded and not closed in general. 
If one views the boundary maps $\Gamma_0$ and $\Gamma_1$ as abstract
Dirichlet and Neumann trace maps then the values of the Weyl function can be interpreted as abstract analogues of the Dirichlet-to-Neumann map in the theory of elliptic PDEs.
For $\lambda,\mu\in\rho(A_0)$ and $h\in\ran\Gamma_0$ we note the useful identities
\begin{equation}\label{okfine}
\gamma(\lambda)^*=\Gamma_1(A_0-\overline{\lambda})^{-1}
\end{equation}
and
\begin{equation}\label{g1}
 \gamma(\lambda)h=(A_0-\mu)(A_0-\lambda)^{-1}\gamma(\mu)h
\end{equation}
as well as
\begin{equation}\label{m1}
 M(\lambda)h=M(\mu)^*h+(\lambda-\mu)\gamma(\mu)^*(A_0-\mu)(A_0-\lambda)^{-1}\gamma(\mu)h
\end{equation}
for the $\gamma$-field and Weyl function, and refer the reader for more details and proofs of the above identities to \cite{BL07,BL12}.

The following theorem from \cite{BL07,BL12} contains a Krein type resolvent formula and provides a criterion to show selfadjointness of the extension $A_\Theta$
in~\eqref{atti}.

\begin{thm}\label{ratebitte}
 Let $S$ be a densely defined, closed, symmetric operator in a Hilbert space $(\sH, \langle \cdot, \cdot \rangle_\sH)$ and let
 $\{\cG,\Gamma_0,\Gamma_1\}$ be a quasi boundary triple for $\overline T=S^*$ with $A_0= T\upharpoonright\ker\Gamma_0$ and $\gamma$-field $\gamma$
 and Weyl function $M$. Let $\Theta$ be an operator in $\cG$ and let 
 $$A_\Theta=T\upharpoonright\ker(\Gamma_1-\Theta\Gamma_0).$$ Assume, in addition,  that $\lambda\in\rho(A_0)$ is not an eigenvalue of $A_\Theta$ or, equivalently, $\ker(\Theta-M(\lambda))=\{0\}$.
 Then the following assertions hold.
\begin{enumerate}
\item $u\in\ran(A_\Theta-\lambda)$ if and only if $\gamma(\overline{\lambda})^*u\in\dom(\Theta-M(\lambda))^{-1}$.
\item For all $u\in\ran(A_\Theta-\lambda)$ one has
\begin{equation}\label{kreini}
(A_\Theta-\lambda)^{-1}u=(A_0-\lambda)^{-1}u+\gamma(\lambda)\bigl(\Theta-M(\lambda)\bigr)^{-1}\gamma(\overline{\lambda})^*u.
\end{equation}
\end{enumerate}
In particular, if $\Theta$ is a symmetric operator in $\cG$ and $\ran \gamma(\overline{\lambda})^*$ is contained in $\dom(\Theta-M(\lambda))^{-1}$ for some $\lambda\in\C^+$ and some $\lambda\in\C^-$ then
$A_\Theta$ is selfadjoint in $\sH$ and the resolvent formula \eqref{kreini} holds for all $\lambda\in\rho(A_\Theta)\cap\rho(A_0)$ and all 
$u\in\sH$.
\end{thm}

Next we provide a slightly generalized variant of the representation formula for the scattering matrix from \cite{BMN15}.  
Let again $S$ be a densely defined, closed, symmetric operator in a Hilbert space $(\sH, \langle \cdot, \cdot \rangle_\sH)$ and let
 $\{\cG,\Gamma_0,\Gamma_1\}$ be a quasi boundary triple for $\overline T=S^*$ with $A_0= T\upharpoonright\ker\Gamma_0$ and $\gamma$-field $\gamma$
 and Weyl function $M$. Assume, in addition, that the extension
 \begin{equation*}
  A_1=T\upharpoonright\ker\Gamma_1
 \end{equation*}
is selfadjoint in $\sH$; in general $A_1$ is only symmetric in $\sH$ and not necessarily closed.
Denote the absolutely continuous subspaces of $A_0$ and $A_1$ by
$\sH^{\rm ac}(A_0)$ and $\sH^{\rm ac}(A_1)$, respectively, let $P^{\rm ac}(A_0)$ be the orthogonal projection onto  $\sH^{\rm ac}(A_0)$ and
let 
$$A_0^{\rm ac}=A_0\upharpoonright \bigl(\dom A_0\cap\sH^{\rm ac}(A_0)\bigr)$$ 
in $\sH^{\rm ac}(A_0)$ be the absolutely continuous part of $A_0$.
If the difference of the resolvents of $A_0$ and $A_1$ is a trace class operator, that is,
\begin{equation}\label{tracechen}
 (A_1-\lambda)^{-1}-(A_0-\lambda)^{-1}\in\sS_1(\sH)
\end{equation}
for some, and hence for all, $\lambda\in\rho(A_0)\cap\rho(A_1)$ then the wave operators
\begin{equation*}
  W_\pm(A_0,A_1):=s-\lim_{t\rightarrow\pm\infty} e^{itA_1} e^{-itA_0} P^{\rm ac}(A_0)
\end{equation*}
exist and satisfy $\ran W_\pm(A_0,A_1)=\sH^{\rm ac}(A_1)$ according to the Birman--Krein theorem \cite{BK}. It follows that the scattering operator 
$$S(A_0,A_1):=W_+(A_0,A_1)^*W_-(A_0,A_1)$$ 
is unitary in the absolutely continuous subspace $\sH^{\rm ac}(A_0)$ of $A_0$, and that $S(A_0,A_1)$ is unitarily equivalent to a multiplication
operator $\{S(\lambda)\}_{\lambda\in\R}$ in a spectral representation of 
the absolutely continuous part $A_0^{\rm ac}$ of $A_0$. The family $\{S(\lambda)\}_{\lambda\in\R}$  is called the {\it scattering matrix} of the pair $\{A_0,A_1\}$;
cf. \cite{BW83,K,RS79,Y92}.

In general the underlying closed symmetric operator $S$ is not simple (or completely non-selfadjoint) 
and hence its selfadjoint part is reflected in the scattering matrix of $\{A_0,A_1\}$.
More precisely, if $S$ is not simple then
there is a nontrivial orthogonal decomposition of the Hilbert space $\sH=\sH_1\oplus\sH_2$ such that
\begin{equation}\label{sssdeco}
 S=S_1\oplus S_2,
\end{equation}
where $S_1$ is a simple symmetric operator in $\sH_1$ and $S_2$ is a selfadjoint operator in $\sH_2$.
Since $A_0$ and $A_1$ are selfadjoint extensions of $S$ there exist selfadjoint extensions
$B_0$ and $B_1$ of $S_1$ in $\sH_1$ such that
\begin{equation}\label{bb}
 A_0=B_0\oplus S_2\qquad\text{and}\qquad A_1=B_1\oplus S_2.
\end{equation}
In the following let $L^2(\R,d\lambda,\cH_\lambda)$ be a spectral representation of the absolutely continuous part $S_2^{\rm ac}$ of the selfadjoint operator $S_2$ in $\sH_2$.

Now we can formulate a variant of \cite[Theorem 3.1 and Corollary 3.3]{BMN15} which is suitable for our purposes. Instead of generalized boundary triples the result is stated 
for quasi boundary triples here. 

\begin{thm}\label{scatti}
Let $S$ be a densely defined, closed, symmetric operator in $\sH$ decomposed in the form \eqref{sssdeco} and let $\{\cG,\Gamma_0,\Gamma_1\}$ be a quasi 
boundary triple for $\overline T=S^*$ with $A_0= T\upharpoonright\ker\Gamma_0$ and $\gamma$-field $\gamma$ and Weyl function $M$. Assume that the extension 
$A_1=T\upharpoonright\ker\Gamma_1$ is selfadjoint in $\sH$ and let $B_0$ and $B_1$ be selfadjoint operators as in \eqref{bb}. Furthermore, suppose that 
 \begin{equation*}
 \overline{\gamma(\lambda_0)}\in\sS_2(\cG,\sH)\quad\text{for some}\,\,\,\,\lambda_0\in\rho(A_0),
 \end{equation*}
 and that 
 $M(\lambda_1)^{-1}$ is a bounded operator in $\cG$ for some $\lambda_1\in\rho(A_0)\cap\rho(A_1)$.
 Then \eqref{tracechen} is satisfied for all $\lambda\in\rho(A_0)\cap\rho(A_1)$ and the following assertions hold.
\begin{enumerate}
\item $\overline{\text{\rm Im}\, M(\lambda)}\in\sS_1(\cG)$ for all $\lambda\in\rho(A_0)$ and
the limit 
$$\overline{\text{\rm Im}\, M(\lambda+i0)} := \lim_{\varepsilon\searrow 0 }\overline{\text{\rm Im}\, M(\lambda+i\varepsilon)}$$
exists in $\sS_1(\cG)$ for a.e. $\lambda\in\R$.
\item For all $\varphi\in\ran\Gamma_0$ and a.e. $\lambda\in\R$ the limit
$$M(\lambda \pm i0)\varphi := \lim_{ \varepsilon\searrow 0 }M(\lambda \pm i\varepsilon)\varphi$$
exists and the operators $M(\lambda \pm i0)$ are closable with boundedly invertible closures $\overline{M(\lambda \pm i0)}$.
\item The space $L^2(\R,d\lambda,\cG_\lambda\oplus\cH_\lambda)$, where
 \begin{equation*}
  \cG_\lambda:=\overline{\ran\bigl(\overline{\text{\rm Im}\, M(\lambda+i0)}\bigr)}\quad \text{for a.e.}\,\,\,\lambda\in\R,
 \end{equation*}
forms a spectral representation of $A_0^{\rm ac}$.
 \item The scattering matrix $\{S(\lambda)\}_{\lambda\in\R}$ of the scattering system $\{A_0,A_1\}$
acting in the space $L^2(\R,d\lambda,\cG_\lambda\oplus\cH_\lambda)$ admits the representation
\begin{equation*}
 S(\lambda)=\begin{pmatrix} S'(\lambda) & 0 \\ 0 & I_{\cH_\lambda}\end{pmatrix}
\end{equation*}
for a.e. $\lambda\in\R$, where
\begin{equation*}
 S'(\lambda)=I_{\cG_\lambda}-2i\sqrt{\overline{\text{\rm Im}\, M(\lambda+i0)}}\,\left(\overline{M(\lambda+i0)}\right)^{-1} \sqrt{\overline{\text{\rm Im}\, M(\lambda+i0)}}
\end{equation*}
is the scattering matrix of the scattering system $\{B_0,B_1\}$.
\end{enumerate}
\end{thm}

In the following we show how the objects of this manuscript fit in the abstract scheme of quasi boundary triples. Let $\Opfree$ be the selfadjoint Laplacian in $L^2(\R^3)$
with domain $H^2(\R^3)$ and let $\Op$ be the Schr\"{o}dinger operator with a $\delta$-interaction of strength $\tfrac{1}{\alpha}$ supported on $\Sigma$ from Definition \ref{schdef}.
Consider the  symmetric operator 
\begin{equation}\label{sssi}
 Su=-\Delta u,\qquad\dom S=\bigl\{u\in H^2(\R^3):u\vert_\Sigma=0\bigr\},
\end{equation}
and define the operator $T$ in $L^2(\R^3)$ by
\begin{equation}\label{ttt}
 Tu=-\Delta u-h\delta_\Sigma,\qquad
 \dom T=H^2(\R^3)\,\dot +\,\bigl\{\gamma_\eta h:h\in\dom \overline{B_\eta}\bigr\}, 
\end{equation}
where $\eta<0$ is chosen such that $0\in\rho(\overline{B_\eta}-\alpha)$ (see Proposition~\ref{Proposition_Main}~(ii) and (iv)) and $\gamma_\eta h=(-\Delta-\eta)^{-1}(h\delta_\Sigma)$ 
is as in \eqref{gammaop}. It follows from the remark below Definition~\ref{def:genTrace} that the sum in the definition of $\dom T$
is direct. Furthermore, $T$ is a well-defined operator in $L^2(\R^3)$ since for an element $u=u_c+\gamma_\eta h\in\dom T$ with $u_c\in H^2(\R^3)$ 
and $h\in\dom \overline{B_\eta}$ one has
\begin{equation}\label{zuzu}
\begin{split}
 -\Delta u-h\delta_\Sigma&=(-\Delta-\eta)(u_c+\gamma_\eta h)+\eta (u_c+\gamma_\eta h)-h\delta_\Sigma\\
 &=-\Delta u_c +\eta \gamma_\eta h \in L^2(\R^3).
\end{split}
 \end{equation}
Note also that
\begin{equation}\label{kernt}
 \ker(T-\eta)=\bigl\{\gamma_\eta h:h\in\dom \overline{B_\eta}\bigr\}.
\end{equation}

In the next proposition we specify a quasi boundary triple 
$\{L^2(\Sigma),\Gamma_0,\Gamma_1\}$ for the adjoint of the symmetric operator $S$
such that $\Opfree=T\upharpoonright\ker\Gamma_0$. 

\begin{prop}\label{qbtprop}
The operator $S$ in \eqref{sssi} is densely defined, closed and symmetric in $L^2(\R^3)$ and satisfies $S^*=\overline T$ with $T$ in \eqref{ttt}.
The triple $\{L^2(\Sigma),\Gamma_0,\Gamma_1\}$, where
\begin{equation}\label{qbtttt}
 \Gamma_0 u= h\quad\text{and}\quad\Gamma_1 u= u_c\vert_\Sigma,\qquad u=u_c+\gamma_\eta h\in\dom T,
\end{equation}
is a quasi boundary triple for $S^*$ such that $\ran\Gamma_0=\dom \overline{B_\eta}$, 
\begin{equation}\label{ddx}
\Opfree=T\upharpoonright\ker\Gamma_0\quad\text{and}\quad \Op=T\upharpoonright \ker\bigl(\Gamma_1-(\alpha-\overline{B_\eta})\Gamma_0\bigr).
\end{equation}
The $\gamma$-field and Weyl function corresponding to $\{L^2(\Sigma),\Gamma_0,\Gamma_1\}$ are given by
\begin{equation}\label{gm}
 \gamma(\lambda) h = (-\Delta-\lambda)^{-1}(h\delta_\Sigma)
\end{equation}
and
\begin{equation}\label{mmi}
M(\lambda) h =\bigl[\bigl((-\Delta-\lambda)^{-1}-(-\Delta-\eta)^{-1}\bigr) h\delta_\Sigma\bigr]\vert_\Sigma,
\end{equation}
for all $\lambda\in\C\setminus[0,\infty)$ and $h\in\dom \overline{B_\eta}$. The values $M(\lambda)$ of the Weyl function are densely defined bounded operators in $L^2(\Sigma)$.
\end{prop}

\begin{proof}
In order to show that the mappings in \eqref{qbtttt} yield a quasi boundary triple for $S^*$ we make use of Theorem~\ref{thm:A2}.
Note first that the identities
\begin{equation*}
 S=T\upharpoonright\bigl(\ker\Gamma_0\cap\ker\Gamma_1\bigr)\quad\text{and}\quad \Opfree=T\upharpoonright\ker\Gamma_0
\end{equation*}
hold. Hence it remains to check that the Green identity 
\begin{equation}\label{ggg}
\langle Tu,v\rangle_{L^2 (\R^3)} - \langle u,Tv\rangle_{L^2 (\R^3)}
=
\langle\Gamma_1 u,\Gamma_0 v\rangle_{L^2(\Sigma)} - \langle\Gamma_0 u,\Gamma_1 v\rangle_{L^2(\Sigma)}
\end{equation}
holds for all $u,v\in\dom T$
and that the range of the mapping $u\mapsto (\Gamma_0 u,\Gamma_1 u)^\top$ is dense in $L^2(\Sigma)\times L^2(\Sigma)$.
In order to verify \eqref{ggg} decompose $u,v\in\dom T$ 
in the form $u=u_c+\gamma_\eta h$ and $v=v_c+\gamma_\eta k$, where $u_c,v_c\in H^2(\R^3)$ and $h,k\in\dom \overline{B_\eta}$.
With the help of \eqref{zuzu} one computes 
\begin{equation*}
 \begin{split}
 \langle & Tu,v\rangle_{L^2 (\R^3)}  - \langle u,Tv\rangle_{L^2 (\R^3)} \\
 &\quad = \bigl\langle T(u_c+\gamma_\eta h),v_c+\gamma_\eta k\bigr\rangle_{L^2 (\R^3)} -\bigl\langle u_c+\gamma_\eta h,T(v_c+\gamma_\eta k)\bigr\rangle_{L^2 (\R^3)}\\
 &\quad=\bigl\langle -\Delta u_c+\eta\gamma_\eta h,v_c+\gamma_\eta k\bigr\rangle_{L^2 (\R^3)} - \bigl\langle u_c+\gamma_\eta h,-\Delta v_c+\eta\gamma_\eta k\bigr\rangle_{L^2 (\R^3)} \\
 &\quad=\langle -\Delta u_c,\gamma_\eta k\rangle_{L^2 (\R^3)} + \langle\eta\gamma_\eta h,v_c\rangle_{L^2 (\R^3)}  \\
 &\qquad\qquad\qquad\qquad\qquad\qquad
  - \langle u_c,\eta\gamma_\eta k\rangle_{L^2 (\R^3)} - \langle\gamma_\eta h,-\Delta v_c\rangle_{L^2 (\R^3)}\\
 &\quad=\bigl\langle (-\Delta -\eta)u_c,\gamma_\eta k\bigr\rangle_{L^2 (\R^3)} - \bigl\langle\gamma_\eta h,(-\Delta -\eta)v_c\bigr\rangle_{L^2 (\R^3)} \\
 &\quad=\langle u_c, k\delta_\Sigma \rangle_{2,-2}-\langle h\delta_\Sigma, v_c\rangle_{-2,2}\\
 &\quad=\langle u_c\vert_\Sigma,k\rangle_{L^2(\Sigma)}-\langle h,v_c\vert_\Sigma\rangle_{L^2(\Sigma)},
 \end{split}
\end{equation*}
which shows \eqref{ggg}. Next assume that for some $\varphi,\psi\in L^2(\Sigma)$ 
\begin{equation*}
 0=\langle\varphi,\Gamma_0 u\rangle_{L^2(\Sigma)}+\langle \psi,\Gamma_1 u\rangle_{L^2(\Sigma)}=\langle\varphi,h\rangle_{L^2(\Sigma)}+
 \bigl\langle\psi,u_c\vert_\Sigma\bigr\rangle_{L^2(\Sigma)}
\end{equation*}
holds for all $u=u_c+\gamma_\eta h\in\dom T$. Restricting to elements $u$ in $H^2(\R^3)$ (i.e. $h=0$) it follows that $\psi=0$.
Finally, if $0=\langle\varphi,h\rangle_{L^2(\Sigma)}$ for all $h\in\dom \overline{B_\eta}$ then $\varphi=0$ as $\overline{B_\eta}$ is densely defined in $L^2(\Sigma)$.
Now it follows from Theorem~\ref{thm:A2} that $\overline T=S^*$ and that $\{L^2(\Sigma),\Gamma_0,\Gamma_1\}$ is a quasi boundary triple for $S^*$. 

In order to see that $\Op=T\upharpoonright\ker(\Gamma_1-(\alpha-\overline{B_\eta})\Gamma_0)$ holds, suppose first that $\Gamma_1 u=(\alpha-\overline{B_\eta})\Gamma_0u$
or, equivalently, $u_c|_\Sigma=(\alpha-\overline{B_\eta})h$ for some
$u=u_c+\gamma_\eta h\in\dom T$. 
Then it follows from the definition of $u |_\Sigma$ in \eqref{traceudef} that
\begin{equation*}
 u\vert_\Sigma=u_c |_\Sigma + (\gamma_\eta h) |_\Sigma=u_c\vert_\Sigma+\overline{B_\eta} h =\alpha h
\end{equation*}
and hence $h=\frac{1}{\alpha} u\vert_\Sigma$. Together with \eqref{ttt} and Definition~\ref{schdef} this implies 
$$\ker \bigl(\Gamma_1-(\alpha-\overline{B_\eta})\Gamma_0\bigr) \subset \dom (\Op)$$ 
and $\Op u = T u$ for all $u \in \ker (\Gamma_1-(\alpha-\overline{B_\eta})\Gamma_0)$. If, conversely, $u \in \dom (\Op)$ then $u = u_c + \gamma_\eta h$ 
for some $u_c \in H^2 (\R^3)$ and some $h \in \dom \overline{B_\eta}$, in particular, $u \in \dom T$. Moreover, 
\begin{align*}
 Tu=- \Delta u - h \delta_\Sigma \in L^2 (\R^3) 
\end{align*}
and
\begin{align*}
\Op u=- \Delta u - \frac{1}{\alpha} u |_\Sigma \cdot \delta_\Sigma \in L^2 (\R^3),
\end{align*}
which implies $(h - \frac{1}{\alpha} u |_\Sigma) \delta_\Sigma \in L^2 (\R^3)$ and thus $h - \frac{1}{\alpha} u |_\Sigma = 0$. Using again the definition of $u |_\Sigma$ 
in \eqref{traceudef}
we obtain
\begin{align*}
 0 = u |_\Sigma - \alpha h = u_c |_\Sigma + (\gamma_\eta h) |_\Sigma - \alpha h = u_c|_{\Sigma} - (\alpha-\overline{B_\eta}) h
\end{align*}
and thus $u \in \ker (\Gamma_1-(\alpha-\overline{B_\eta})\Gamma_0)$. 
The second identity in \eqref{ddx} follows.

Next it will be shown that the $\gamma$-field and Weyl function corresponding to $\{L^2(\Sigma),\Gamma_0,\Gamma_1\}$ have the form in \eqref{gm} and \eqref{mmi}. Note first
that \eqref{kernt} and the definition of $\Gamma_0$ imply $\gamma(\eta)h=\gamma_\eta h=(-\Delta-\eta)^{-1}(h\delta_\Sigma)$ for all $h\in\ran\Gamma_0=\dom\overline{B_\eta}$. 
Furthermore, for $\lambda\in\C\setminus [0,\infty)$
we conclude from \eqref{g1} and \eqref{ddx} that
\begin{equation*}
 \gamma(\lambda)h=(\Opfree-\eta)(\Opfree-\lambda)^{-1}\gamma(\eta)h=(-\Delta-\lambda)^{-1}(h\delta_\Sigma) 
\end{equation*}
holds. Moreover,
\begin{equation}\label{goodtohave2}
 \gamma(\lambda)^*u=\Gamma_1(\Opfree-\overline{\lambda})^{-1}u=\bigl((\Opfree-\overline{\lambda})^{-1}u\bigr)\vert_\Sigma
\end{equation}
for all $u\in L^2(\R^3)$ by \eqref{okfine}; cf. \eqref{goodtohave}.
It follows from  the definition of $\Gamma_1$ that
\begin{equation*}
 M(\eta) h=\Gamma_1 \gamma(\eta) h=\Gamma_1(-\Delta-\eta)^{-1}(h\delta_\Sigma)=0
\end{equation*}
holds for all $h\in\ran\Gamma_0=\dom\overline{B_\eta}$.
From \eqref{m1} and \eqref{goodtohave2} we then conclude for $\lambda\in\C\setminus [0,\infty)$ and $h\in\ran\Gamma_0=\dom\overline{B_\eta}$
\begin{equation*}
 \begin{split}
  M(\lambda)h&=(\lambda-\eta)\gamma(\eta)^*(\Opfree-\eta)(\Opfree-\lambda)^{-1}\gamma(\eta)h\\
             &=\bigl[(\lambda-\eta)(\Opfree-\lambda)^{-1}(-\Delta-\eta)^{-1} h\delta_\Sigma\bigr]\vert_\Sigma\\
             &=\bigl[\bigl((-\Delta-\lambda)^{-1}-(-\Delta-\eta)^{-1}\bigr) h\delta_\Sigma\bigr]\vert_\Sigma;
 \end{split}
\end{equation*}
cf.~\eqref{eq:uc}. We have shown that \eqref{mmi} holds. 
Note also that $M(\eta)=0$ and \eqref{m1} with $\mu=\eta$ imply that the operators $M(\lambda)$ are bounded. This completes the proof of Proposition~\ref{qbtprop}.
\end{proof}

\begin{rem}\label{remchen}
If the operator $T$ in \eqref{ttt} is replaced by the operator
\begin{equation*}
 T'u=-\Delta u-h\delta_\Sigma,\qquad
 \dom T'=H^2(\R^3)\,\dot +\,\bigl\{\gamma_\eta h:h\in L^2(\Sigma)\bigr\}, 
\end{equation*}
then $T\subset T'$ and the assertions in Proposition~\ref{qbtprop} remain valid with $T$ replaced by $T'$ and $\dom\overline{B_\eta}$ replaced
by $L^2(\Sigma)$, respectively. In particular, in this situation the boundary map $\Gamma_0$ maps onto $L^2(\Sigma)$ and hence the quasi boundary
triple $\{L^2(\Sigma),\Gamma_0,\Gamma_1\}$ in Proposition~\ref{qbtprop} is a generalized boundary triple, and the values $M(\lambda)$ of the 
Weyl function are bounded operators defined on $L^2(\Sigma)$. It follows from \eqref{eq:uc} and \eqref{eq:integral2} that 
\begin{equation*}
(M(\lambda)h)(x)=\int_\Sigma h(y)\frac{e^{i\sqrt{\lambda}\vert x - y\vert}-e^{i\sqrt{\eta}\vert x - y\vert}}{4\pi \vert x - y\vert}\,d\sigma(y),\quad x \in \Sigma, \quad h\in L^2(\Sigma).
\end{equation*}
Note, however, that 
$\Gamma_1$ is not surjective and $\{L^2(\Sigma),\Gamma_0,\Gamma_1\}$ is not an ordinary boundary triple. 
\end{rem}


\begin{thebibliography}{99}

\bibitem{AS} M.~Abramowitz and I.~Stegun, Handbook of Mathematical Functions with Formulas, Graphs, and Mathematical Tables, U.S.\ Government Printing Office, Washington, D.C., 1964.

\bibitem{A90} M.\,S.~Agranovich, {\it Elliptic operators on closed manifolds}, Partial Differential
Equations~VI, Encyclopaedia Math. Sci., vol. 63, Springer, Berlin (1994), 1--130.

\bibitem{AG93} N.\,I.~Akhiezer and I.\,M.~Glazman, Theory of Linear Operators in Hilbert Space, Dover publications, 1993.

\bibitem{AGHH} S.~Albeverio, F.~Gesztesy, R.~H{\o}egh-Krohn, and H.~Holden, Solvable Models in Quantum Mechanics. With an Appendix by Pavel Exner. 2nd ed., AMS Chelsea Publishing, Providence, RI, 2005.



\bibitem{AGS87}  J.-P.~Antoine, F.~Gesztesy, and J.~Shabani, {\it Exactly solvable models of sphere interactions in quantum mechanics}, J.\ Phys.\ A~20 (1987), 3687--3712.

\bibitem{BW83} H.~Baumg\"artel and M.~Wollenberg, Mathematical Scattering Theory, Akademie-Verlag, Berlin, 1983.


\bibitem{BL07} J.~Behrndt and M.~Langer, {\it Boundary value problems for elliptic partial differential operators on bounded domains}, J.\ Funct.\ Anal.\ 243 (2007), 536--565.

\bibitem{BL12} J.~Behrndt and M.~Langer, {\it Elliptic operators, Dirichlet-to-Neumann maps and quasi boundary triples}, London Math.\ Soc.\ Lecture Note Series 404 (2012), 121--160.

\bibitem{BLL13} J.~Behrndt, M.~Langer, and V.~Lotoreichik, {\it Schr\"odinger operators with $\delta$ and $\delta'$-potentials supported on hypersurfaces}, Ann.\ Henri Poincar\'e 14 (2013), 385--423.

\bibitem{BMN08} J.~Behrndt, M.\,M.~Malamud, and H.~Neidhardt, {\it Scattering matrices and Weyl functions}, Proc.\ Lond.\ Math.\ Soc.\ 97 (2008), 568--598.


\bibitem{BMN15} J.~Behrndt, M.\,M.~Malamud, and H. Neidhardt, {\it Scattering matrices and Dirichlet-to-Neumann maps}, arXiv:1511.02376.

\bibitem{BP35} H.~Bethe and R.~Peierls, {\it Quantum theory of the diplon}, Proc.\ R.\ Soc.\ Lond.,\ Ser.\ A~148 (1935), 146--156.



 
\bibitem{BDE03} F.~Bentosela, P.~Duclos, and P.~Exner, \textit{Absolute continuity in periodic thin tubes and strongly coupled leaky wires}, Lett.\ Math.\ Phys.~65 (2003), 75--82. 

\bibitem{BIN79} O.\,V.~Besov, V.\,P.~Il'in, and S.\,M.~Nikol'skii, Integral Representations of Functions and Imbedding Theorems, Vol. II, Scripta Series in Mathematics, Washington, D.C.: V.H. Winston \& Sons. New York etc.: John Wiley \& Sons, 1979.

\bibitem{BK} M.\,Sh.~Birman and M.\,G.~Krein, \textit{On the theory of wave operators and scattering operators}, Soviet. Math. Dokl. 3 (1962), 740--744.

\bibitem{BSS01}  M.\,Sh.~Birman, T.\,A.~Suslina, and R.\,G.~Shterenberg, {\it Absolute continuity of the two-dimensional Schr\"odinger operator with delta potential concentrated on a periodic system of curves} (Russian), Algebra i Analiz 12 (2000), 140--177; translation in St. Petersburg Math.\ J.~12 (2001), 983--1012.
%

\bibitem{BL77} A.\,S.~Blagove{\v{s}}{\v{c}}enskii and K.\,K.~Lavrent'ev, \textit{A three-dimensional Laplace operator with a boundary condition on the real line} (in Russian), Vestn.\ Leningr.\ Univ., Mat.\ Mekh.\ Astron.~1 (1977), 9--15.

\bibitem{BEKS94} J.\,F.~Brasche, P.~Exner, Yu.\,A.~Kuperin, and P.~\v{S}eba, {\it Schr\"odinger operators with singular interactions},  J.\ Math.\ Anal.\ Appl.\ 184 (1994), 112--139.

\bibitem{BT92}  J.\,F.~Brasche and A.~Teta, \textit{Spectral analysis and scattering theory for Schr\"odinger operators with an interaction supported by a regular curve}, Ideas and methods in quantum and statistical physics. In memory of Raphael H{\o}egh-Krohn (1938-1988). Volume 2, Cambridge University Press (1992), 197--211. 

\bibitem{BD06} R.~{Brummelhuis} and P.~{Duclos}, {\it Effective Hamiltonians for atoms in very strong magnetic fields},
J.\ Math.\ Phys.~47 (2006), 032103.

\bibitem{BGP08}  J.~Br\"uning, V.~Geyler, and K.~Pankrashkin, 
{\it Spectra of self-adjoint extensions and applications to solvable Schr\"odinger operators}, 
Rev. Math. Phys. 20 (2008), 1--70.

\bibitem{CDFMT12} M.~Corregi, G.~Dell'Antonio, D.~Finca, A.~Michelangeli, and A.~Teta, {\it Stability for a system of $N$ fermions plus a different particle with zero-range interactions}, Rev.\ Math.\ Phys.~24 (2012), 1250017.

\bibitem{DFT94} G.~Dell'Antonio, R.~Figari, and A.~Teta, {\it Hamiltonians for systems of $N$ particles interacting through point interactions}, 
Ann. Inst. H. Poincar\'e Phys.\ Theor.~60 (1994), 253--290.

\bibitem{DM91} V.\,A.~Derkach and M.\,M.~Malamud,
{\it Generalized resolvents and the boundary value problems for Hermitian operators with gaps},
J.\ Funct.\ Anal. 95 (1991), 1--95.

\bibitem{DM95} V.\,A.~Derkach and M.\,M.~Malamud,
{\it The extension theory of Hermitian operators and the moment problem},
J.\ Math.\ Sci.\ (New York) 73 (1995), 141--242.

\bibitem{E05} P.~Exner, {\it An isoperimetric problem for leaky loops and related mean-chrod inequalities}, J.\ Math.\ Phys.~46 (2005), 062105.

\bibitem{E08} P.~Exner, \textit{Leaky quantum graphs: a review}, Analysis on Graphs and its Applications. Selected papers based on the Isaac Newton Institute for Mathematical Sciences programme, Cambridge, UK, 2007, Proc.\ Symp.\ Pure Math.~77 (2008), 523--564.

\bibitem{EF07} P.~Exner and R.\,L.~Frank, \textit{Absolute continuity of the spectrum for periodically modulated leaky wires in ${\mathbb{R}^{3}}$}, Ann.\ Henri Poincar\'e~8 (2007), 241--263.

\bibitem{EHL06} P.~Exner, E.\,M.~Harrell, and M.~Loss, \textit{Inequalities for means of chords, with application to isoperimetric problems}, Lett.\ Math.\ Phys.~75 (2006), 225--233.

\bibitem{EI01} P.~Exner and T.~Ichinose, {\it Geometrically induced spectrum in curved leaky wires}, J.\ Phys.\ A~34 (2001), 1439--1450.

\bibitem{EK02} P.~Exner and S.~Kondej, \textit{Curvature-induced bound states for a $\delta$ interaction supported by a curve in $\mathbb R^3$}, Ann.\ Henri Poincar\'e~3 (2002), 967--981.

\bibitem{EK04} P.~Exner and S.~Kondej, \textit{Strong-coupling asymptotic expansion for Schr\"odinger operators with a singular interaction supported by a curve in $\mathbb R^3$}, Ann.\ Henri Poincar\'e~16 (2004), 559--582.

\bibitem{EK08} P.~Exner and S.~Kondej, \textit{Hiatus perturbation for a singular Schr\"odinger operator with an interaction supported by a curve in $\mathbb R^{3}$}, J.\ Math.\ Phys.~49 (2008), 032111.

\bibitem{EK15} P.~Exner and S.~Kondej, \textit{Strong coupling asymptotics for Schr\"odinger operators with an interaction supported by an open arc in three dimensions}, 
Rep.\ Math.\ Phys.~77 (2016), 1--17.

\bibitem{EH15} P.~Exner and H.~Kova{\v{r}}{\'{\i}}k, Quantum Waveguides, Springer, Heidelberg, 2015.

\bibitem{EY01} P.~Exner and K.~Yoshitomi, {\it Band gap of the Schr\"odinger operator with a strong $\delta$-interaction on a periodic curve}, Ann.\ Henri Poincar\'e~2 (2001), 1139--1158.

\bibitem{EY02} P.~Exner and K.~Yoshitomi, {\it Persistent currents for 2D Schr\"odinger operator with a strong $\delta$-interaction on a loop}, J.\ Phys.\ A~35 (2002), 3479--3487.

\bibitem{F36} E.~Fermi, {\it Sul moto dei neutroni nelle sostanze idrogenate}, Ric.\ Sci.\ Progr.\ Tecn.\ Econom.\ Naz.~2 (1936), 13--52.

\bibitem{GK69} I.\,C.~Gohberg and M.\,G.~Krein, Introduction to the Theory of Linear Nonselfadjoint Operators, Transl.\ Math.\ Monogr., vol. 18., Amer. Math. Soc., Providence, RI, 1969.

\bibitem{GKP89} R.\,L.~Graham, D.\,E.~Knuth, and O.~Patashnik, Concrete Mathematics, Addison-Wesley Publishing Company, 1989.

\bibitem{GR} I.\,S.~Gradshteyn and I.\,M.~Ryzhik, Table of Integrals, Series, and Products, Elsevier/Academic Press, Amsterdam, 2015.

\bibitem{FK98} A.~Figotin and P.~Kuchment, {\it Spectral properties of classical waves in high-contrast periodic media}, SIAM J.\ Appl.\ Math.~58 (1998), 683--702.


\bibitem{K} T.~Kato, Perturbation Theory for Linear Operators, Springer-Verlag, Berlin, 1995. 

\bibitem{K02} S.~Kondej, {\it On the eigenvalue problem for self-adjoint operators with singular perturbations}, Math.\ Nachr.~244 (2002), 150--169.

\bibitem{K12} S.~Kondej, \textit{Resonances induced by broken symmetry in a system with a singular potential}, Ann. Henri Poincar\'e~13 (2012), 1451--1467.



\bibitem{KP31} R.\,de\,L.~Kronig and W.~Penney, {\it Quantum mechanics of electrons in crystal lattices},
Proc.\ Roy.\ Soc.\ Lond.~130 (1931), 499--513.

\bibitem{K78} Y.\,V.~Kurylev, \textit{Boundary conditions on a curve for a three-dimensional Laplace operator} (in Russian), Zap.\ Nauchn.\ Sem.\ Leningrad.\ Otdel.\ Mat.\ Inst.\ Steklov. (LOMI)~78 (1978), 112--127.

\bibitem{K83} Y.\,V.~Kurylev, \textit{Boundary conditions on curves for the three-dimensional Laplace operator}, J.\ Sov.\ Math.~22 (1983), 1072--1082.




\bibitem{LL61} E.\,H.~Lieb and W.~Liniger, {\it Exact analysis of an interaction Bose gas. I: The general solution and the ground state}, Phys.\ Rev.~130 (1963), 1605--1616.

\bibitem{MO16}
A.~Michelangeli and A.~Ottolini, \textit{On point interaction realised as Ter-Martirosyan-Skornyakov Hamiltonians}, arXiv:1606.05222.

\bibitem{McL}
W.~McLean, Strongly Elliptic Systems and Boundary Integral Equations,
Cambridge University Press, 2000.

\bibitem{M11} R.\,A.~Minlos, {\it On point-like interaction between $n$ fermions and another particle}, Mosc.\ Math.\ J.~11 (2011), 113--127.

\bibitem{MF62} R.\,A.~Minlos and L.\,D.~Faddeev, {\it On the point interaction for a three-particle system in quantum mechanics}, Soviet Physics Dokl.~6 (1962), 1072--1074.

\bibitem{P01} A.~Posilicano, \textit{A Krein-like formula for singular perturbations of self-adjoint operators and applications}, J.\ Funct.\ Anal.~183 (2001), 109--147.

\bibitem{R70} W.~Rudin, Real and Complex Analysis, New York, McGraw-Hill, 1970.
%
%

\bibitem{S95} Yu.~Shondin, \textit{On the semiboundedness of delta-perturbations of the Laplacian on curves with angular points}, Theor.\ Math.\ Phys.~105 (1995), 1189--1200.


\bibitem{RS75} M.~Reed and B.~Simon, Methods of Modern Mathematical Physics. II: Fourier Analysis, Self-Adjointness, Academic Press, New York, 1975.


\bibitem{RS79} M. Reed and B. Simon, Methods of Modern Mathematical Physics. III. Scattering
Theory, Academic Press, New York--London, 1979.
\bibitem{ST56}
G.\,V.~Skorniakov and K.\,A.~Ter-Martirosian, 
\textit{Three body problem for short range forces. I. Scattering of low energy neutrons by deuterons}, Sov.\ Phys.\ JETP 4 (1956), 648--661.



\bibitem{T90} A.~Teta, \textit{Quadratic forms for singular perturbations of the Laplacian}, Publ.\ Res.\ Inst.\ Math.\ Sci.~26 (1990), 803--817.

\bibitem{T35} L.\,H.~Thomas, {\it The interaction between a neutron and a proton and the structure of $H^3$}, Phys.\ Rev.~47 (1935), 903--909.

\bibitem{Y92} D.\,R.~Yafaev, Mathematical Scattering Theory: General Theory, Translations of Mathematical Monographs 105, American Mathematical Society, Providence, RI, 1992.

\end{thebibliography}
\end{document}